\documentclass[english,12pt]{smfart}
\usepackage{amssymb}
\usepackage{amsfonts}
\usepackage{amscd}
\usepackage[T1]{fontenc}
\usepackage{color}
\usepackage{setspace}
\usepackage{tikz}

\newbox\removebox
\newcommand\remove[1]{%
\setbox\removebox=\ifmmode\hbox{$#1$}\else\hbox{#1}\fi%
\leavevmode
\rlap{\textcolor{black}{\vrule height0.8ex depth-0.6ex width\wd\removebox}}%
\box\removebox
}
\long\def\bigremove#1{%
\par\setbox\removebox=\vbox{#1}%
\vbox{%
\vbox to0pt{\hbox{\tikz\draw[color=black,thick] (0,0) -- (\wd\removebox,-\ht\removebox)  (\wd\removebox,0) -- (0,-\ht\removebox);}}
\box\removebox
}
}

\usepackage{mathrsfs} 

\def\VF{\mathrm{VF}}
\def\VG{\mathrm{VG}}

\def\deg{\operatorname{deg}}
\def\ac{{\overline{\rm ac}}}

\def\LPas{\cL_{\rm DP}}
\def\Lan{\cL_{\rm an}}

\def\LT{{\rm LT}}

\def\longhookrightarrow{\mathrel\lhook\joinrel\longrightarrow}
\let\cal\mathcal

\def\11{{\mathbf 1}}
\def\AA{{\mathbb A}}

\def\CC{{\mathbb C}}

\def\FF{{\mathbb F}}

\def\NN{{\mathbb N}}

\def\QQ{{\mathbb Q}}
\def\RR{{\mathbb R}}

\def\ZZ{{\mathbb Z}}

\def\cA{{\mathcal A}}

\def\cF{{\mathcal F}}

\def\cL{{\mathcal L}}
\def\cM{{\mathcal M}}

\def\cO{{\mathcal O}}

\def\cT{{\mathcal T}}

\def\cX{{\mathcal X}}

\def\cZ{{\mathcal Z}}
\def\llp{\mathopen{(\!(}}
\def\llb{\mathopen{[\![}}

\def\rrp{\mathopen{)\!)}}
\def\rrb{\mathopen{]\!]}}

\def\Fie{\mathrm{Field}}

\newtheorem*{thmblock}{Theorem}

\newtheorem{thm}[subsubsection]{Theorem}
\newtheorem{lem}[subsubsection]{Lemma}
\newtheorem{cor}[subsubsection]{Corollary}
\newtheorem{prop}[subsubsection]{Proposition}

\newtheorem{question}[subsubsection]{Question}

\theoremstyle{definition}
\newtheorem{defn}[subsubsection]{Definition}
\newtheorem{example}[subsubsection]{Example}

\newtheorem{def-prop}[subsubsection]{Proposition-Definition}
\newtheorem{def-theorem}[subsubsection]{Theorem-Definition}
\newtheorem{def-lem}[subsubsection]{Lemma-Definition}

\theoremstyle{remark}
\newtheorem{remark}[subsubsection]{Remark}

\theoremstyle{plain}

\numberwithin{equation}{subsection}

\renewcommand{\theequation}{\thesubsection.\arabic{equation}}

  {\par\medskip\noindent #1\par\begingroup%
    \advance\leftskip by 1em\advance\rightskip by 1em}%
  {\par\endgroup}

\newcommand{\ord}{\operatorname{ord}}

\begin{document}

\setcounter{tocdepth}{1} 

\title[Non-archimedean parametrizations]{Non-archimedean Yomdin-Gromov parametrizations and points of bounded height}


\author{Raf Cluckers}
\address{Universit\'e Lille 1, Laboratoire Painlev\'e, CNRS - UMR 8524, Cit\'e Scientifique, 59655
Villeneuve d'Ascq Cedex, France, and,
Katholieke Universiteit Leuven, Department of Mathematics,
Celestijnenlaan 200B, B-3001 Leu\-ven, Bel\-gium}
\email{Raf.Cluckers@math.univ-lille1.fr}
\urladdr{http://math.univ-lille1.fr/$\sim$cluckers}

\author{Georges Comte}
\address{Universit\'e Savoie Mont Blanc, LAMA,
CNRS UMR 5127,
F-73000 Chamb\'ery, France}
\email{Georges.Comte@univ-savoie.fr}
\urladdr{gc83.perso.sfr.fr}

\author{Fran\c cois Loeser}
\address{Sorbonne Universit\'es, UPMC Univ Paris 06, UMR 7586 CNRS, Institut Math\'ematique de Jussieu, F-75005 Paris, France}
\email{Francois.Loeser@upmc.fr}
\urladdr{http://www.math.jussieu.fr/$\sim$loeser/}

\begin{abstract}
We prove an analogue of the Yomdin-Gromov Lemma for $p$-adic definable sets and more broadly in a non-archimedean, definable context. This analogue keeps track of piecewise approximation by Taylor polynomials, a nontrivial aspect in the totally disconnected case. We apply this result to bound the number of rational points of bounded height on the transcendental part of $p$-adic subanalytic sets, and to bound the dimension of the set of complex polynomials of bounded degree lying on an algebraic variety defined over $\mathbb{C}\llp t \rrp$, in analogy to results by Pila and Wilkie, resp.~by Bombieri and Pila. Along the way we prove, for definable functions in a general context of non-archimedean geometry, that local Lipschitz continuity implies piecewise global Lipschitz continuity.
\end{abstract}

\maketitle

\begin{spacing}{0.999}
\section{Introduction}\label{sec:1}

\subsection{}A very efficient tool in diophantine geometry is the so-called
determinant method  which was developed by Bombieri and Pila in the influential paper
\cite{BP} about the number of integral points of bounded
height on affine algebraic and transcendental plane curves. Basically, the method
consists in using a determinant of a suitable
set of monomials evaluated at the integral points, in order to construct a family of auxiliary polynomials
vanishing at all integral points on the curve within a small enough box.
Building on the estimates  in \cite{BP} for algebraic curves, Pila proved in \cite{pila_ast} bounds
on the number of integral (resp. rational) points of bounded height on affine (resp. projective) algebraic varieties of any  dimension,
improving on previous results by S. D. Cohen using the large sieve method \cite{cohen}. Important further improvements going  towards optimal bounds conjectured by Serre in \S 13 of \cite{serre}
have been made since by Heath-Brown,  Browing  and Salberger \cite{HB},\cite{BHS},\cite{sal}.

In \cite{PiWi}, Pila and Wilkie proved a general estimate for the number of rational points on the transcendental part of sets definable in an o-minimal structure; this has
been used in a spectacular way by  Pila to provide
an  unconditional proof of some cases of the Andr\'e-Oort Conjecture \cite{PiAnnals} (see also \cite{ScaBourbaki}, \cite{ScaBull} and \cite{ScaPanoramas} for surveys on applications in diophantine geometry of the Pila-Wilkie Theorem).
Lying at the heart  of Pila and Wilkie's approach is  the possibility
of having uniform - in terms of number of parametrizations and in terms of bounds on the partial derivatives -  $C^k$-parametrizations. These parametrizations are provided by
an o-minimal version of Gromov's algebraic
parametrization Lemma \cite{gromov} (see also \cite{Burguet}), itself a refinement of a previous result  of
Yomdin \cite{YY},\cite{YY2}. Such $C^k$-parametrizations enter the determinant method via Taylor approximation.

The aim of this paper is to provide a version of the Yomdin-Gromov Lemma in the non-archimedean setting, notably for subanalytic sets over $\mathbb{Q}_p$ and
$\mathbb{C} \llp t \rrp$, and to develop
the determinant method in this context in order to obtain non-archimedean analogues of some of the results in diophantine geometry mentionned above.
At first sight one may have doubts about the realizability of
such a   program, since because of the totally disconnected character of non-archimedean spaces,  it seems there is no way for a  global Taylor formula to make sense   in this framework.
A first indication that the situation may not be completely hopeless, is provided by the fact that in previous work \cite{CCL} (see also \cite{ch}) we have been able to prove a version of
first-order Taylor approximation, piecewise globally, in the definable $p$-adic setting. In the present paper, though we extend this first order result to a much wider situation,
we have chosen not trying to  generalize it   to higher order,
but instead we  show directly the existence of
uniform $C^k$-parametrizations that do  satisfy Taylor approximation, which is enough for our purpose.
The existence of such parametrizations is provided by Theorem \ref{GYT} which is the main result of  Section \ref{sec:2}.
In Section \ref{sec:3}, we deduce  a $p$-adic analogue of the Theorem of Pila and Wilkie in \cite{PiWi}, in the  strengthened version given by Pila in \cite{PiSelecta} in terms of blocks.
In Section \ref{sec:4}, we prove a
geometric analogue of results of Bombieri-Pila \cite{BP} and Pila \cite{pila_ast} over
$\mathbb{C} \llp t \rrp$ where  counting number of points is replaced by counting dimensions.

 The diophantine applications we give in the $p$-adic case (concerning the density of rational points on the transcendental part of definable sets) and the motivic case (concerning the density of rational points on algebraic sets) are quite different.
One should notice that in the algebraic case, working over $\mathbb{Q}_p$ instead of $\CC\llp t \rrp$ would not provide better estimates than those following directly from the archimedean ones in
\cite{BP} and  \cite{pila_ast}.

\subsection{}\label{ssec:1.2}Let us spell out in some more detail
 basic versions of three of our results for subanalytic $p$-adic sets.

One calls a set $X\subset \QQ_p^n$ semialgebraic if it is definable in the ring language with parameters from $\QQ_p$. By
adding to the ring language symbols for analytic functions, one obtains subanalytic sets (see Section \ref{sec:GYT} below, with $L=\QQ_p$).
The dimension of a nonempty subanalytic set $X\subset \QQ_p^n$ is the largest integer $m\geq 0$ such that there exists a coordinate projection $\pi:\QQ_p^n\to \QQ_p^m$ such that $\pi(X)$ has nonempty interior. We will simply denote by $\vert x \vert $ the $p$-adic norm of an element $x\in \QQ_p$ and
when furthermore $x\in \QQ$, we will denote by $\vert x \vert_\RR$ the real norm of $x$.


For $X$ a subset of $\QQ_p^n$ and $T>1$ a real number, write $X(\QQ,T)$
for the set consisting of points $(x_1,\cdots, x_n)$ in $X\cap \QQ^n$ such that one can write $x_i$ as $a_i/b_i$
where $a_i$ and $b_i\not=0$ are integers with $|a_i|_\RR\leq T$ and $|b_i|_\RR\leq T$.


For $X$ a subset of $\QQ_p^n$, write $X^{\rm alg}$ for the subset of $X$ consisting of points $x$ such that there exists an algebraic curve $C\subset \AA_{\QQ_p}^n$ such that $C(\QQ_p)\cap X$ is locally of dimension $1$ at $x$.

With this notation, the following statement is a particular case of
Theorem \ref{blocksQ}:

\begin{thmblock}
Let $X\subset \QQ_p^n$ be a subanalytic set of dimension $m$ with $m<n$. Let $\varepsilon>0$ be given. Then there exist an integer $C=C(\varepsilon,X)>0$ and a semialgebraic set $W= W(\varepsilon,X)\subset \QQ_p^n$ such that $W\cap X$ lies inside $X^{\rm alg}$, and such that
for each $T$,
one has
$$
\# (X\setminus W) (\QQ,T) \leq C T^\varepsilon .
$$
\end{thmblock}


More generally, we also provide estimates for algebraic points of bounded degree on $X$, as follows.
For a rational number $a$, we define $H_0(a)$ as $\max (\vert r\vert_\RR, \vert s\vert_\RR)$ when $a=r/s$ with $r$ and $s$ integers which are either relatively prime or are such that $s=1$. For nonrational $a$, $H_0(a)$ is defined to be $+\infty$. We extend $H_0$ to tuples $a=(a_i)$ by putting $H_0(a) = \max_i(H_0(a_i))$.
 For an integer $k\geq 1$ and any $x\in\QQ_p$ 
we define $H_k^{\rm poly}(x)$ as 
$\min_a (H_0(a))$
where the minimum runs over all nonzero tuples $a=(a_i)_{i\in \{0,\ldots,k\}} $ such that $\sum_{i=0}^k a_i x^k = 0$ if such tuple exists, and as $+\infty$ otherwise. One extends $H_k^{\rm poly}$ to $x$ in $\QQ_p^n$ by taking the maximum of the $H_k^{\rm poly}(x_i)$ for $i=1,\ldots,n$.
 For $X$ a subset of $\QQ_p^n$, $k\geq 0$ an integer and $T>1$ a real number, write
$X(k,T)$
for the set consisting of $x$ in $X$ satisfying that $H_k^{\rm poly}(x)$ is at most equal to $T$.
The following statement follows from Theorem \ref{algheights}:
\begin{thmblock}
Let $X\subset \QQ_p^n$ be a subanalytic set of dimension $m$ with $m<n$. Let $\varepsilon>0$ and an integer $k\geq 0$ be given. Then there exist an integer $C=C(\varepsilon,k,X)>0$ and a semialgebraic set $W= W(\varepsilon,k,X)\subset \QQ_p^n$ such that $W\cap X$ lies inside $X^{\rm alg}$, and such that
for each $T$,
one has
$$
\# (X\setminus W) (k,T) \leq C T^\varepsilon.
$$
\end{thmblock}

Our proofs of  Theorems  \ref{algheights} and \ref{blocksQ} rely on the existence of
reparametrizations provided by Theorem \ref{GYT}, where we prove the following:

\begin{thmblock}
Let $n\geq 0$, $m\geq 0$ and $r\geq 0$ be integers and let $X\subset \ZZ_p^n$ be a subanalytic set of dimension $m$.
Then there exists a finite collection of subanalytic functions $g_{i}:P_i\subset \ZZ_p^m\to X$ such that the union of the $g_i(P_i)$
equals $X$, the $g_i$ have $C^r$-norm bounded by $1$, and the $g_i$ are approximated by their Taylor polynomials of degree $r-1$ with
remainder of order $r$, globally on $P_i$.
\end{thmblock}

For precise definitions of the $C^r$-norm and  approximation by Taylor polynomials of certain degree  with controlled remainder,  we refer to  Definition \ref{Tr}. Note  that  a key point in this non-archimedean statement is that the approximation holds globally on $P_i$, which
represents a challenging goal because of total disconnectedness.
On the opposite, in the real case, it is enough to consider convex charts, since for such charts global Taylor approximation is immediate.
In the core of the paper all these results will be stated and proved uniformly for definable families.   

\subsection{}\label{ssec:1.3}
We end this introduction with a quick overview of our results  over  the base field $\CC\llp t \rrp$.
In this case, the analogue of Theorem  \ref{GYT} essentially still holds, except one has to
replace ``finite'' by ``parametrized by a constructible subset of $\mathbb{C}^s$ for some $s$''.

For each positive integer $r$ we denote by
$\mathbb{C} [t]_{< r}$ the  set of complex polynomials of degree
$<r$.
For any subset  $A$  of
$\mathbb{C} \llp t \rrp^n$
we denote by $A_r$ the set
$A \cap (\mathbb{C} [t]_{< r})^n$ and
by $n_r(A)$ the dimension of the Zariski closure of
$A_r$ in
$(\mathbb{C} [t]_{< r})^n \simeq \mathbb{C}^{nr}$.
When  $X$ is  an algebraic subvariety of
$\mathbb{A}^n_{\mathbb{C} \llp t \rrp}$ of dimension $m$, for every positive integer $r$ one has  the  basic estimate
$n_r (X) \leq r m$ (cf. Lemma \ref{trivest}). Assume $X$ is  irreducible of degree $d$.
The  main result of Section \ref{sec:4},
Theorem \ref{motivicBP}, states
 that,
for every positive integer $r$, one has
$$
n_r (X) \leq r (m - 1)  + \Bigl\lceil\frac{r}{d}\Bigr\rceil,
$$
which is an improvement on the basic estimate as soon as $X$ is not linear.

This result can be seen as an instance of motivic point counting, like in
the paper \cite{MHZ}, where in a somewhat different context a detailed study of the motivic height zeta function leads to asymptotic estimates for dimensions of spaces of sections.
It is a motivic analogue of Pila's results of \cite{pila_ast}.
As in
\cite{pila_ast}, we reduce to the case of plane curves
by Lang-Weil type arguments. However, to prove the result in this case, we do not follow the original method of Bombieri and Pila in \cite{BP}, which seems difficult to adapt
in this setting.  We follow instead  a strategy   introduced by Marmon in \cite{Marmon}, which relies on  the  Yomdin-Gromov Lemma in place of the elaborate analytic arguments used in  \cite{BP}.

\subsection{Some shortcuts}

Although the general results on Lipschitz continuity of Theorems \ref{Lip} and \ref{Lipmixed} are used to prove Theorems \ref{GYT} and \ref{GYTs}, as far as our applications to points of bounded height are concerned, they are not needed in full generality.
First, in the $p$-adic setting of Section \ref{sec:3},
Theorems \ref{Lip} and \ref{Lipmixed} are not new since they are proved in \cite{CCL} and even appear in a slightly sharper form in \cite{ch}.
Secondly, in Section \ref{sec:4} when working over $\CC\llp t\rrp$, only the one-dimensional case of Theorem \ref{GYT}, namely with $m=1$, is used. This case of Theorem \ref{GYT} only relies on the one-dimensional case of Theorem \ref{Lip}, which can be proved similarly as the one-dimensional case of \cite{CCL}. Such one-dimensional cases require less work and are easier to prove than the general cases. For example, for the case of general dimension for Theorems \ref{Lip} and \ref{Lipmixed} one cannot use definable sections as in \cite{CCL} and we had to find the alternative approach via Theorem \ref{Lipcenter}; in the one-dimensional case definable sections were not used in \cite{CCL}.
Another simplification, in the $p$-adic case, would be to exploit the fact that the residue rings $\ZZ_p/(p^n)$ are finite. Indeed, this finiteness allows one to zoom and scale as in the real case, and this can serve as an alternative
to the passage to the algebraic closure of $L$ when proving theorems \ref{GYT} and \ref{GYTs} for $L=\QQ_p$. This zoom and scale technique would not work in the case of $K=\CC\llp t\rrp$, unless one is content in Theorem \ref{motivicBP} with a weakening of the upper bound to $r (m - 1) + r/d + c_d$, for some constant depending on $d$.

\subsection*{Acknowledgments}
\hspace{0.5cm}We would like to thank Antoine Chambert-Loir, Immanuel Halupczok and Ehud Hrushovski for stimulating discussions related to this work. In particular, Antoine Chambert-Loir
 directed us to Marmon's paper, Immanuel Halupczok provided an example that finitely many pieces do not suffice in general for Theorem \ref{Lip}, and Ehud Hrushovski encouraged us to
consider potential diophantine applications.

The authors were supported in part by the European Research Council under the European Community's Seventh Framework Programme (FP7/2007-2013) with ERC Grant Agreement numbers 246903 NMNAG and 615722 MOTMELSUM, and by the Labex CEMPI  (ANR-11-LABX-0007-01). We would also like to thank the IHES  and the FIM of the ETH in Z\"urich, where part of the research was done. The paper was finalized while the authors participated in the MSRI program: Model Theory, Arithmetic Geometry and Number Theory.

\section{Piecewise Lipschitz continuity in tame non-archimedean geometry}\label{sec:5}

In this section and in the next one we prove   non-archimedean analogues of the Yomdin-Gromov parametrization lemma.  In the same time, we prove that our parametrizations can be uniformly approximated by their Taylor polynomials. The ability to approximate  parametrizations by their Taylor polynomials with small error terms is key for counting points of bounded height in the non-archimedean case, as well as  in the real case, although in the real case this approximation is classical when the domain is convex and boundeof  order $\ge 1$ (see Theorem \ref{GYT}).


The Yomdin-Gromov parametrization lemma, as generalized by Pila and Wilkie in \cite{PiWi}, gives for any bounded definable set $X$ of dimension $m$ in $\RR^n$ (in an o-minimal structure on $\RR$) and any integer $r>0$ a finite collection of definable maps $f_i:[0,1]^m\to X$ whose ranges together cover $X$ and whose $C^r$-norms are bounded by $1$. By convexity of $[0,1]^m$, and techniques based on the mean value theorem, one can control the error terms when approximating $f_i$ by a Taylor polynomial of degree $r-1$. Both convexity and the mean value theorem do not carry to our context. For $r>1$, we do not know, even for $p$-adic semialgebraic functions $f$ on $\ZZ_p^m$, whether having small $C^r$-norm allows one to piecewise control the error term, globally on each piece, when one approximates $f$ by a Taylor polynomial of degree $r-1$. In the case where $r=1$, the desired approximation coincides with Lipschitz continuity, and the study of piecewise Lipschitz continuity, in a general non-archimedean context, is the content of Sections \ref{sec:5b} and \ref{sec:mix}. From Section \ref{sec:GYT} on, we will place ourselves in a more concrete framework of definable sets in complete, discretely valued fields (possibly with extra restricted analytic functions in the language), to treat $C^r$-parametrizations with good Taylor approximation when $r>1$.

\subsection{Lipschitz continuity in tame non-archimedean geometry}\label{sec:5b}

In \cite{CCL}, piecewise Lipschitz-continuity for a semialgebraic or subanalytic function $f:X\subset \QQ_p^n\to \QQ_p$ is shown to hold whenever $f$ is locally Lipschitz continuous with a fixed Lipschitz constant. Moreover, the pieces can be taken to be definable. In \cite{ch}, the Lipschitz constants were further controlled in an optimal way when going from local to global on each piece. In this section, we extend the result of \cite{CCL} in two ways, namely to many new structures with a non-archimedean geometry, including $\CC\llp t \rrp$, and to other languages than the semialgebraic and subanalytic ones, including some weaker languages without multiplication.
When the residue field is not finite, one is led to replace finite definable partitions by  definable families with parameters running over the residue field.
Our study of Lipschitz continuity is subdivided into two cases:  equicharacteristic zero  and  mixed characteristic. Both are axiomatically treated.
In the mixed characteristic case, residue rings, and not only the residue field, are used.

\medskip

We first introduce the set-up   adapted to the equicharacteristic zero case.
The typical example to have in mind is that of henselian valued field $K$ of equicharacteristic zero with (multiplicatively written) value group $\Gamma^{\times}$ and residue field $k$. Then in  Section \ref{sec:mix} we will consider the mixed characteristic case.

\medskip

Let $\Gamma=\Gamma^\times \cup \{0\}$ be the disjoint union of a nontrivial ordered abelian group $\Gamma^\times$ with a minimal element $\{0\}$, where the group operation on $\Gamma^\times$ is written multiplicatively, and where we put $0 \cdot g=g \cdot 0=0$ for all $g\in \Gamma$. Recall that an ordered abelian group is an abelian group with a total order $<$ such that $a<b$ implies $ac<bc$ for all elements $a,b,c$ of the group. Such a group is automatically torsion free, and hence, the order $<$ has no endpoints.
Let $K$ be an additively written abelian group and suppose we are given a surjective map $|\cdot|:  K\to \Gamma: x\mapsto |x| $ with the following properties for all $x,y\in K$
\begin{itemize}
\item[$\bullet$] $|x|=0$ if and only if $x=0$,
\item[$\bullet$] $|x|=|-x|$,
\item[$\bullet$]
$
|x+y| \leq \max (|x| , |y| ),
$
\item[$\bullet$] if $|x|>|y|$, then $|x+y| = |x|$.
\end{itemize}

An open ball is by definition a subset $B\subset K$ of the form $\{x\in K\mid |x-a|< \gamma \}$ for some $\gamma \in \Gamma^\times$ and $a\in K$; such $\gamma$ is unique and is called the radius of the open ball $B$ (not to be confused with the radii of closed balls defined in Section \ref{sec:2}). Since $\Gamma^\times$ has no endpoints, each open ball is an infinite set.

Consider a set $k$ containing a special element $0$ and write $k^\times$ for $k\setminus \{0\}$. Suppose that we are given a surjective map $\ac:K\to k$ with $\ac{\ }^{-1}(0)=\{0\}$ and such that for each $\xi\in k^\times $ and $\gamma \in \Gamma^\times$, the set
$$
\{t\in K\mid \ac(t)=\xi,\ |t|=\gamma\}
$$
is an open ball of radius $\gamma$.
Let us more generally introduce the notation
$$
A_{\xi,\gamma} := \{t\in K\mid \ac(t)=\xi,\ |t|=\gamma\},\ \mbox{ for
$\xi\in k$ and $\gamma\in \Gamma$.}
$$
Note that the family of sets $A_{\xi,\gamma}$ is a disjoint family whose union equals $K$ when $\xi$ varies in $k$ and $\gamma$ in $\Gamma$. Clearly $
A_{0,0}$ equals $\{0\}$, and both $A_{0,\gamma}$ and $A_{\xi,0}$ are empty for nonzero $\gamma$ and nonzero $\xi$.

We put on $K$ the valuation topology, that is, the topology with the collection of open balls as base, and
the product topology on Cartesian powers of $K$. 
Note that $K$ thus becomes a topological group.
For a tuple $x=(x_1, \cdots, x_n)\in K^n$, $|x|$ stands for $\max_{i\in \{1, \cdots, n\}} |x_i|$.

Next we recall the definition of Lipschitz continuity and we define a special variant of continuity, called s-continuity.

\begin{defn}\label{def:lip}
Let a function $f: X \to K$ be given, with $X\subset K^n$. 
For $\gamma\in \Gamma^\times$, the  function $f: X\subset K^n \to K$
is called $\gamma$-Lipschitz (globally on $X$) if for all $x$ and
$y$ in $X$,
$$
|f(x) -f(y)| \leq \gamma | x - y| .
$$
The function $f$ is called locally $\gamma$-Lipschitz
if every point of $ X$ has a neighbourhood on which $f$ is $\gamma$-Lipschitz.

\end{defn}

\begin{defn}[s-continuity]\label{defjacprop}
Let $F:A\to K$ be a function for some set $A\subset K$. Say that $F$
is s-continuous if for each open ball $B\subset A$ the set $F(B)$ is either a singleton or an open ball, and,
there exists $\gamma= \gamma (B)\in \Gamma$ such that
\begin{equation}\label{s-gamma}
|F(x)-F(y)| = \gamma |x-y|\ \mbox{ for all $x,y\in B$}.
\end{equation}
\end{defn}

If a function $g:U\subset K^n\to K$ on an open $U$ is $s$-continuous in, say, the variable $x_n$, by which we mean that $g(a,\cdot)$ is $s$-continuous for each choice of $a=(x_1,\ldots,x_{n-1})$ then we write $|\partial g/\partial x_n ( a,x_n)|$ for the element $\gamma\in \Gamma$ witnessing the s-continuity of $g(a,\cdot)$ locally at $x_n$, namely, $\gamma$ satisfies (\ref{s-gamma}) for the function $F(\cdot)=g(a,\cdot)$, where $x,y$ run over some ball $B$ containing $x_n$ such that $\{a\}\times B\subset U$.

\par

Note that for an $s$-continuous function $F:A\subset K\to K$ on an open $A$ such that $F$ is moreover $1$-Lipschitz, one has $|\partial F(x)/\partial x| \leq 1$ for all $x\in A$. Hence, for such $F$, for $x\in A$ with $|\partial F(x)/\partial x| >0$, and for any ball $B\subset A$ containing $x$, say, of radius $r_B$, the set $F(B)$ is a ball of radius $\leq r_B$. Moreover, for compositions of s-continuous functions one has a certain form of the chain rule which corresponds to the classical chain rule for differentiation, cf.~Lemma \ref{cor:invers} and its proof.

\medskip

Let $\cL_{\rm Basic}$ be the first order language with the sorts $K$, $k$ and $\Gamma$, and symbols for addition on $K$, for $\ac:K\to k$, $|\cdot|:K\to\Gamma$, and for the order and the multiplication on $\Gamma$.
Let $\cL$ be any expansion of $\cL_{\rm Basic}$. 
By $\cL$-definable we mean $\emptyset$-definable in the language $\cL$, and likewise for other languages than $\cL$.
Write $K^0=\{0\}$, $k^0=\{0\}$, and $\Gamma^0=\{0\}$, with a slight abuse of notation.
Note that $\cL$ may have more sorts than $\cL_{\rm Basic}$, since it is an arbitrary expansion.
\begin{example}
This language $\cL_{\rm Basic}$ is very basic (since it does not have multiplication), and can be interpreted in many structures.
We give an example of a triple $(K,k,\Gamma)$ with $\cL_{\rm Basic}$-structure.
Let $K$ be the Laurent series field $\FF_p\llp t \rrp $, seen as a group for addition, put $\Gamma^\times := 2^{\ZZ}$, let $|\cdot|$ be the $t$-adic norm with $|t|=2^{-1}$ on $K$, $k$ the finite field $\FF_p$, and let $\ac$ send a nonzero Laurent series $a(t)$ to the coefficient of its lowest degree nonzero term. A more natural example of an $\cL$-structure with $\cL$ being $\cL_{\rm Basic}$ together with multiplication on $K$, is for the field $K=\CC\llp t\rrp$ with $t$-adic norm and $\ac$ defined as for $\FF_p\llp t \rrp $.
\end{example}

\begin{defn}[Tame configurations]\label{config}
Given integers $a\geq 0$, $b\geq 0$, a set
$$
T\subset K\times k^a\times \Gamma^b,
$$
and some $c\in K$, say that $T$ is in $c$-config if there is $\xi\in k$ such that
$T$ equals the union over $\gamma\in \Gamma$ of sets
$$(c + A_{\xi,\gamma})\times U_{\gamma}$$
for some $U_{\gamma}\subset  k^a\times \Gamma^b$.
If moreover $\xi\not=0$ we speak of an open $c$-config, and if $\xi=0$ we speak of a graph $c$-config.
If $T$ is nonempty and in $c$-config, then $\xi$ and the sets $U_\gamma$ such that $A_{\xi,\gamma}$ is nonempty are uniquely determined by $T$ and $c$.

Say that $T \subset K\times k^a\times \Gamma^b$ is in $\cL$-tame config if there exist $s \geq 0$ and $\cL$-definable functions
$$g:K\to k^s\ \mbox{ and }\ c:k^s\to K\
$$
such that the range of $c$ contains no open ball,
and, for each $\eta\in k^s$, the set $$T\cap (g^{-1}(\eta) \times k^a \times \Gamma^b  )$$
is in $c(\eta)$-config.
\end{defn}


By the aforementioned uniqueness in the nonempty case, one sees that,
for an $\cL$-definable set $T$ which is in $c$-config, the collection of sets $U_{\gamma}$ can be taken to be an $\cL$-definable family.
The functions $g$ and $c$ used in a $\cL$-tame config are in general not unique, but still one often calls $c$ the center (of the configuration).



\begin{defn}\label{test}
For any $\cL$-structure $M$ which is elementarily equivalent to $(K,\cL)$ and for any language $L$ which is obtained from $\cL$ by adding some elements of  $M$ (of any sort) as constant symbols, call $(M,L)$ a test pair for $(K,\cL)$.
\end{defn}

\begin{defn}[Tameness]\label{Ktame}
Say that $(K,\cL)$ is weakly tame if the following conditions hold.
\begin{enumerate}
\item[(1)] Each $\cL$-definable set $T\subset K\times k^a\times \Gamma^b$ with $a\geq 0$, $b\geq 0$ is in $\cL$-tame config.

\item[(2)] 
For any $\cL$-definable function $F:X\subset K\to K$ there exist $s\geq 0$ and an $\cL$-definable function
$g:X\to k^s$ such that, for each $\eta\in k^s$, the restriction of $F$ to $g^{-1}(\eta)$ is s-continuous.
\end{enumerate}
Say that $(K,\cL)$ is tame when each test pair $(M,L)$ for $(K,\cL)$ is weakly tame.
Call an $\cL$-theory $\cT$ tame if for each model $\cM$ of $\cT$, the pair $(\cM,\cL)$ is tame.
\end{defn}

Condition (2) is a substitute for the so-called Jacobian property which holds for henselian valued fields in equicharacteristic zero equipped with the Denef-Pas language.
We
refer to Theorem 6.3.7 of \cite{CLip} for a closely related Jacobian property; it can be adapted to the Denef-Pas language
using
 Theorem 4.1 of \cite{Pas1} on elimination of valued field quantifiers.
For henselian valued fields in  mixed characteristic  equipped with the generalized Denef-Pas language, see Section \ref{sec:mix}.
In \cite{Hal-W}, Definition 2.19 and Theorem 5.12, one will find a version of the Jacobian property in higher dimensions.
Some examples of tame structures are provided in Section \ref{ex-tame}.

We can now state our first main result on Lipschitz continuity, going from local to piecewise global on parts parametrized by variables running over $k$.

\begin{thm}\label{Lip}
Suppose that $(K,\cL)$ is tame.
Let $f:X\subset K^n\to K$ be an $\cL$-definable function which is locally $1$-Lipschitz. Then there exists an $\cL$-definable function
$$
g:X\to k^s
$$
for some $s\geq 0$ such that for each $\eta\in k^s$, the restriction of $f$ to $g^{-1}(\eta)$ is $1$-Lipschitz.
\end{thm}


Theorem \ref{Lip} is complemented by Theorem \ref{Lipcenter} about simultaneous partitions of domain and range into parts with $1$-Lipschitz centers. This is an improvement of Proposition 2.4 of \cite{CCL}, where this is done for the domain only, and only in the $p$-adic case.

For $h:D\subset A\times B\to C$ any function between sets and for $a\in A$, write $D_a$ for the set $\{b\in B\mid (a,b)\in D\}$ and write $h(a,\cdot)$ or $h_a$ for the function which sends $b\in D_a$ to $h(a,b)$. We use similar notation $D_a$ and $h(a,\cdot)$ or $h_a$ when $D$ is a Cartesian product $\prod_{i=1}^n A_i$ and $a\in p(D)$ for some coordinate projection $p:D\to \prod_{i\in I\subset \{1,\cdots, n\}} A_i$.

\begin{thm}[Lipschitz continuous centers in domain and range]\label{Lipcenter}
Suppose that $(K,\cL)$ is tame.
Let $f:A\subset K^n\to K$ be an $\cL$-definable function which is locally $1$-Lipschitz. Then, for a finite partition of $A$ into definable parts, the following holds for each part $X$.
There exist $s\geq 0$, a coordinate projection $p:K^n\to K^{n-1}$ and
$\cL$-definable functions
$$
g:X\to k^s,\
c:k^s\times K^{n-1}\to K \mbox{ and } d:k^s\times K^{n-1}\to K
$$
such that, for each $\eta\in k^s$, the restrictions of $c(\eta,\cdot)$ and $d(\eta,\cdot)$ to $p(g^{-1}(\eta))$ are $1$-Lipschitz and, for each $w$ in $p(K^{n})$, the set $g^{-1}(\eta)_w$ is in $c(\eta,w)$-config and the image of $g^{-1}(\eta)_w$ under $f_w$ is in $d(\eta,w)$-config.
\end{thm}

Note that the projection $p$ in Theorem \ref{Lipcenter} {\sl a priori} depends on the part $X$.
Theorems \ref{Lip} and \ref{Lipcenter} are proved by a joint induction on $n$. By the improvement of Proposition 2.4 of \cite{CCL} given by Theorem \ref{Lipcenter}, we can avoid the usage of definable sections (Skolem functions), which were heavily used in \cite{CCL}. This is especially helpful since one does not have definable Skolem functions in the general context of tameness. Let us first explain the general strategy of the proofs. We first prove some general, although easy, results about sets and functions in tame structures, from statement \ref{fin} up to \ref{invers}. An analogue of the chain rule for derivation can be used, based on s-continuity. The Lipschitz continuity of  $c$
in Theorem \ref{Lipcenter} is proved as in \cite{CCL}, as well as the case $n=1$ of Theorem \ref{Lip}.
What is new here is that   $d$ in Theorem \ref{Lipcenter} can be  required to be Lipschitz continuous as well. Working piecewise in the proof of Theorem \ref{Lip}, one may, after some triangular transformation, assume that the centers of domain and of range are both zero. In which case
 the comparison of distances in domain and in range becomes easier.

\medskip

We  prove preliminary statements in view of Theorems \ref{Lip} and \ref{Lipcenter}.

\begin{lem}\label{fin}
Suppose that $(K,\cL)$ is tame.
If $h: k^a\times \Gamma^b\to K$ is $\cL$-definable for some $a,b\geq 0$, then the image of $h$ contains no open ball.
\end{lem}
\begin{proof} Let $h: k^a\times \Gamma^b\to K$ be $\cL$-definable and let
$T\subset K\times k^a\times \Gamma^b$ be the graph of $h$, with the natural identification. Now take $g:K\to k^s$ and $c:k^s\to K$
such that the range of $c$ contains no open ball and such that $T_\eta:= T\cap (g^{-1}(\eta)\times  k^a\times \Gamma^b )$
is in $c(\eta)$-config, for any $\eta\in k^s$.
By Definition \ref{config} there exist sets $U_{\gamma,\eta}\subset  k^a\times \Gamma^b$
such that $T_\eta$ equals the union of
$$ (c(\eta) + A_{\xi,\gamma})\times U_{\gamma,\eta}$$
over $\gamma\in G$. From this description as a Cartesian product together with the fact that $T$ is the graph of $h$, it follows that $U_{\gamma,\eta}$ is empty whenever $A_{\xi,\gamma}$ contains more than one element. Moreover, whenever $A_{\xi,\gamma}$ is a singleton one has $A_{\xi,\gamma}=\{0\}$. Hence, the range of $h$ is contained in the range of $c$, which contains no open ball.
\end{proof}

The next proposition has to be compared with the real monotonicity theorem (see \cite{Dri}, (1.2) Chapter 3).

\begin{prop}[Injectivity versus constancy]\label{inj-cons}Suppose that $(K,\cL)$ is tame.
Let $F:X\subset K\to K$ be $\cL$-definable. Then there exist $s\geq 0$ and an $\cL$-definable function
$
g:X\to k^s
$
such that for each $\eta\in k^s$ the restriction of $F$ to $g^{-1}(\eta)$ is injective or constant.
\end{prop}
\begin{proof}
By tameness, logical compactness, and by going to a test pair, it suffices to treat the case where $F$ is s-continuous and moreover locally injective or locally constant.
Indeed, by tameness there exists an $\cL$-definable function $g_{00}:X\to k^{s_0}$ such that $F$ is s-continuous on each fiber of $g_{00}$, and such that each fiber of $g_{00}$ is either a singleton or open in $K$. If we can prove the statement for each restriction of $F$ to $g_{00}^{-1}(\xi)$ for any $\xi\in k^{s_0}$, then we are done by logical compactness, since $(K,\cL)$ is arbitrary in the proposition and since such a restriction is $\cL(\xi)$-definable and lives thus in the test pair $(K,\cL(\xi))$, also a tame structure. Note that such logical compactness yields finitely many candidate definable functions $g_i$ out of infinitely many ones, but these $g_i$ can be combined to a single one by putting $g(x) := g_i(x)$ for the minimal $i$ such that $F$ is either injective or constant on $g_i^{-1}(\eta)$, which is an $\cL$-definable condition.

Let us first suppose that $F$ is locally injective. Let $G(F)\subset K^2$ be the graph of $F$.
By logical compactness\footnote{Logical compactness will be used like in this proof but without extra explanation to go from the one-variable setting in any model and with any constants added, to the family-version.} and tameness, there exist $s\geq 0$ and definable functions
$$
g_0:G(F)\to k^s\ \mbox{ and }\  c: k^s\times F(X)\to K
$$
such that for each value $(\eta,t)\in k^s\times F(X)$, the set $g^{-1}_0(\eta)_t\subset K$ is in $c(\eta,t)$-config.
Indeed, for each value $y$ in $F(X)$ and by tameness of the test pair $(K,\cL(y))$, there exist $\cL(y)$-definable maps $g_y$ from $G(F)\cap K\times \{y\}$ to $k^{s_y}$ and $c_y: k^{s_y}\to K$ such that
for each value $\eta\in k^{s_y}$, the set $g^{-1}_y(\eta)$, considered as a subset of $K$, is in $c_y(\eta)$-config. Logical compactness yields again finitely many $g_i$ and $c_i$ out of these possibly infinitely many $g_y$ and $c_y$, and these can again be combined by defining $g_0(x,t)$ as $g_i(x,t)$ and $c(\eta,t)$ as $c_i(\eta,t)$ for the minimal $i$ such that $g_i^{-1}(\eta)_t$ is in $c_i(\eta,t)$-config.
Set
$$
g: \begin{cases} X\to k^s \\ x\mapsto g_0(x,F(x)).
\end{cases}
$$
By the local injectivity of $F$ and by the definition of being in $c(\eta,t)$-config, it follows that $F^{-1}(t)$ is contained in the range of $c(\cdot,t)$. Fix $\eta\in g(X)$. Then $y\mapsto c(\eta,y)$ is the inverse function of the restriction of $F$ to $g^{-1}(\eta)$, which is thus injective.

Let us finally suppose that $F$ is locally constant.
By tameness, s-continuity, and local constancy of $F$, there exist $a,b\geq 0$ and
$$
h:X\to k^a\times \Gamma^b,
$$
such that $F$ is constant on each fiber of $h$. By Lemma \ref{fin}, $F(X)$ contains no open ball. By tameness there exist definable functions
$$
g_1: F(X)\to k^s\ \mbox{ and }\ c:k^s\to K
$$
such that $g_1^{-1}(\eta)\cap F(X)$ is in $c(\eta)$-config, meaning
 that $F(X)$ is contained in the range of $c$. Now define $g:X\to k^s$ as sending $x$ to the unique $\eta$ with $c(\eta) = F(x)$. It is clear that $F$ is constant on $g^{-1}(\eta)$ for any $\eta\in k^s$.
\end{proof}

By the following corollary, tameness appears as a special variant of $b$-minimality as defined in \cite{CLb}, but tameness has a more geometrical flavor.
Note that for us the sorts $k$ and $\Gamma$ play a rather different role, while in $b$-minimality all sorts other than $K$ are treated on the same footing. In particular parameters in $k$ play a special role for us, defining the splitting of a space into  parts parametrized by the residue field in Theorem \ref{Lip} and elsewhere.

\begin{cor}[$b$-minimality]\label{b}
Suppose that $(K,\cL)$ is tame. Let $\cT$ be the theory of the restriction of $(K,\cL)$ to the sorts $K,k,\Gamma$, that is, $\cT$ is the theory of the structure on these sorts having as definable sets the $\cL$-definable sets. 
Then $\cT$ is $b$-minimal, with main sort $K$ and where the role of balls is played by open balls.
\end{cor}
\begin{proof}
The proof is immediate from Definition 2.1 of \cite{CLb}, Lemma \ref{fin}, and Proposition \ref{inj-cons}.
\end{proof}

In particular, the dimension theory for $b$-minimal structures of \cite{CLb} applies to tame structures. 

\begin{prop}[Continuity]\label{continuity}
Let $f:X\subset K^n\to K$ be $\cL$-definable. Then there exists an $\cL$-definable function
$$
g:X\to k^s
$$
for some $s\geq 0$ such that, for each $\eta\in k^s$, the restriction of $f$ to $g^{-1}(\eta)$ is continuous. Further, for given $T\subset X\times k^a\times \Gamma^b$ with $a\geq 0$, $b\geq 0$, one can moreover take $g$ such that the fibers $T_x$ are locally independent of $x\in g^{-1}(\eta)$ for each $\eta\in k^s$. Moreover, if $n=1$, one can ensure that, for any ball $B$ contained in $X$, $T_x=T_{x'}$ for any $x,x'$ which lie in $B$. Finally, if $X$ is open and $f$ is locally $1$-Lipschitz in each variable separately, then one can moreover take $g$ such that the restriction of $f$ to $g^{-1}(\eta)$ is locally $1$-Lipschitz for any $\eta\in k^s$.
\end{prop}
\begin{proof}
The statements for $n=1$ follow from tameness. Indeed, the statement about continuity follows from s-continuity when $n=1$ and the statement about $T$ follows from
existence of tame configurations. Note that one can preserve continuity of $f$ while controlling $T_x$ since  it is  possible to combine two maps $g_i:X\to k^{s_i}$ into a single map $g=(g_1,g_2)$ refining both $g_1$ and $g_2$, in the sense that each fiber of $g$ is included in a fiber of $g_1$ and of $g_2$.
The statements for general $n$ follow from tameness, induction on $n$, and logical compactness (used in the same way as in the proof of Proposition \ref{inj-cons}).

We shall illustrate this by giving now full  details for the proof of the statement about the continuity of $f$ in the case $n=2$. The proof for $n > 2$ and
and the  proof of the statement about $T_x$ are completely similar.

By compactness, by the case $n=1$, and by going to a test pair, we may suppose that $f(x_1,\cdot)$ and $f(\cdot,x_2)$ are continuous for each $x_1$ and $x_2$ in $K$. Likewise, we may suppose that $X_{x_1}$ is in $c(x_1)$-config, for some definable function $c$ and for each $x_1$ in $K$, and that $p_1(X)$ is in $d$-config for some $d\in K$, where $p_1:K^2\to K$ is the projection $(x_1,x_2)\mapsto x_1$. Similarly, we may further suppose that $c$ is continuous on $p_1(X)$. Again for similar reasons, we may assume that the definable family of sets $U_{x_1,\gamma}$ does locally not depend on $x_1\in K$, where the  $U_{x_1,\gamma}$ are such that, for some $\xi\in k$ and for each $x_1$ in $K$,
$$
X_{x_1} = \cup_{\gamma\in \Gamma} (c(x_1) + A_{\xi,\gamma})\times U_{x_1,\gamma}.
$$
Furthermore, we can also assume that,  for each $x_1$,
\begin{equation}\label{Xcopen}
\mbox{$X_{x_1}$ is in open $c(x_1)$-config and $p_1(X)$ is in open $d$-config}.
 \end{equation}
Indeed, after partitioning, the only case left is when $X_{x_1}$ or $p_1(X)$ are in graph config for each $x_1$, which, by continuity of $c$, reduces to the case $n=1$.


For any function $h:A\subset X\to K$ and any $a\in A$, let $\Delta(h,a)$ be the set of pairs $(\delta,\varepsilon)$ in $(\Gamma^\times)^2$ such that,
for any $b$ in $A$, if  $|b-a| < \delta$, then  $|h(b) - h(a)| < \varepsilon$.

For each $x_1,x_2$ in $K$, consider the sets
\begin{equation}\label{x12c}
\Delta(f(x_1,\cdot),x_2) \mbox{ and }  \Delta(f(\cdot,x_2),x_1). 
\end{equation}
They form definable families of subsets of $(\Gamma^\times)^2$ with parameters $(x_1,x_2)$ in $K^2$.

Now choose $(x_1,x_2)$ in $X$ and  $\varepsilon$ in $\Gamma^\times$. Choose also $\delta\in \Gamma^\times$ such that $(\delta,\varepsilon)$ lies in the intersection of the two sets  in (\ref{x12c}). Such a $\delta$ exists since $f(x_1,\cdot)$ and $f(\cdot,x_2)$ are assumed to be continuous.
By continuity of $c$,  local independence of $x_1$ for the sets $U_{x_1,\gamma}$, and by the openness assumption in (\ref{Xcopen}), we can take $\delta$ so small so that, for any $(v,w)$ in $X$ with $| (x_1,x_2) - (v,w) | < \delta$, one has that $(x_1,w)$ lies in $X$.

Now, for any $(v,w)$ in $X$ with $| (x_1,x_2) - (v,w) | < \delta$, we have
\begin{eqnarray*}
|f(x_1,x_2) - f(v,w)| & = & |f(x_1,x_2) - f(x_1,w) + f(x_1,w) - f(v,w)|  \\
 & \leq & \max \big( |f(x_1,x_2) - f(x_1,w)| , \ | f(x_1,w) - f(v,w)| \big)\\
 &\leq & \varepsilon,
\end{eqnarray*}
which yields the continuity of $f$.
\end{proof}

\begin{lem}[Inverses]\label{invers}
Let $c:X\subset K\to K$ be $\cL$-definable. Then there exists an $\cL$-definable function
$$
g:X\to k^s
$$
for some $s\geq 0$ such that, for each $\eta\in k^s$, either $c$ is locally $1$-Lipschitz on $g^{-1}(\eta)$, or, the restriction of $c$ to $g^{-1}(\eta)$ is injective and its inverse function is locally $1$-Lipschitz.
\end{lem}
\begin{proof}
The statement is clear,  by Proposition \ref{inj-cons} and the definitions of tameness and s-continuity.
\end{proof}

Combining Lemma \ref{invers} with a form of the chain rule for differentiation, we find the following several-variable result.
\begin{cor}\label{cor:invers}
Let $X\subset K^n$ be $\cL$-definable and of dimension $d<n$. Then there exist an $\cL$-definable function
$$
g:X\to k^s,
$$
for some integer  $s$, a finite partition of the graph $G(g)$ of $g$ into $\cL$-definable parts $A_i$,
and for each $i$ an injective coordinate projection $p_i:A_i\to K^d\times k^s$ and $\cL$-definable functions
$$
h_i: p_i(A_i)\to K^{n-d}
$$
such that the union over $i$ of the graphs $G(h_{i})$ equals $G(g)$ and such that the functions $h_{i,\eta}$ are locally $1$-Lipschitz for each $\eta\in k^s$.
\end{cor}
\begin{proof}
By compactness,  working piecewise, and going to a test pair, we may assume that $X$ is already the graph of a function
$$
h:U\subset K^d\to K^{n-d},
$$
which is s-continuous in each variable separately.
By induction on $d$ we may suppose that $U$ is open. We will treat the case  $d=n-1$, the general case being similar. After reordering the variables $x_1,\ldots,x_d$, working piecewise, we may suppose that $|\partial h/\partial x_{d}|$ is maximal among the $|\partial h/\partial x_{i}| $ on the whole of $U$ for $i=1,\ldots,d$, and that $|\partial h/\partial x_{d}|>1$ on $U$. Moreover, by Proposition \ref{inj-cons} and compactness, we may suppose for any $a=(x_1,\ldots,x_{d-1})$ that $h(a,\cdot)$ is injective on $U_{a}$. Now we can reverse the role of $x_n=x_{d+1}$ and $x_d$, by reordering the coordinates. This way $h$ is replaced by a function $\hat h$ sending $(x_1,\ldots,x_{d-1},t)$ to the compositional inverse
$$
h(x_1,\ldots,x_{d-1},\cdot)^{-1}(t).
$$
By s-continuity , with the notation from just below Definition \ref{defjacprop}, we have, for each $i=1,\ldots,d-1$ and for $x,t$ with $h(x)=(x_1,\ldots,x_{d-1},t)$ that
$$
\Bigl |\frac{\partial \hat h(x_1,\ldots,x_{d-1},t)}{\partial x_i}\Bigr| =  \Bigl | \frac{\partial
h(x)}{\partial x_i} \Bigr| \cdot \Bigl |\frac{\partial h(x)}{\partial
x_d}\Bigr|^{-1},
$$
which is at most one by our assumption that $|\partial h/\partial x_{d}|$ is maximal among the $|\partial h/\partial x_{i}|$.
\end{proof}

Now we come to the proof of our main results on Lipschitz continuity.

\begin{remark}\label{rem:comp}
As the proof of Theorem \ref{Lipcenter} for $n=1$ will show, the hypothesis that $f$ is locally $1$-Lipschitz is not needed at all when $n=1$. Hence, Theorem \ref{Lipcenter} for $n=1$ holds even when $f$ is not locally $1$-Lipschitz. Furthermore, when $n=1$,  if $f$ is injective and $g$ is as given by Theorem \ref{Lipcenter}, the function $f$ gives a correspondence between the maximal balls included in $g^{-1}(\eta)$ and the maximal balls included in $f(g^{-1}(\eta))$.
\end{remark}

%

\begin{proof}[Proof of Theorem \ref{Lipcenter} for $n=1$] %
First suppose that $f$ is injective and s-continuous, and that $X$ equals an open ball $c_0+A_{\xi,\gamma}$ for some $\cL$-definable $c_0\in K$, $\xi\in k$ and $\gamma\in \Gamma$. Write $Y$ for $f(X)$. It follows from the case assumptions that $Y$ is an open ball.
By tameness, there exist $s_0\geq 0$ and $\cL$-definable functions
$$
h:Y\to k^{s_0} \mbox{ and } d:k^{s_0}\to K
$$
such that $h^{-1}(\eta)$ is in $d(\eta)$-config for each $\eta\in k^{s_0}$. Define
$$
g: \begin{cases} X\to k^{s_0} \\ x\mapsto h(f(x)),\end{cases}
$$
and define $c(\eta)$ as $c_0$ when $d(\eta)$ lies outside $Y$ and as $f^{-1}(d(\eta))$ when $d(\eta)$ belongs to $Y$. It follows by s-continuity and injectivity of $f$ that $g^{-1}(\eta)$ is in $c(\eta)$-config. The slightly more general case where $f$ is injective and s-continuous and where $X$ is in $c_0$-config for some $\cL$-definable $c_0\in K$ is treated similarly, by choosing $h$ and defining $g$ and $c$ as in the above construction.
Finally we consider the general case.
By tameness and Proposition \ref{inj-cons}, there exist $\cL$-definable functions $g_0:X\to k^{s_0}$ and $c:k^{s_0}\to K$, such that
for each $\eta\in k^{s_0}$, the restriction of $f$ to $g_0^{-1}(\eta)$ is s-continuous, and, injective or constant, and such that the set $g_0^{-1}(\eta)$ is in $c(\eta)$-config. Now we finish the proof by noting that the above construction, applied to the restrictions of $f$ to $g_0^{-1}(\eta)$, works definably and uniformly in $\eta\in g_0(X)$.
\end{proof}

%
%
%

\begin{proof}[Proof of Theorem \ref{Lipcenter} for general $n$] %
We proceed by induction on $n$, assuming that Theorems \ref{Lip} and \ref{Lipcenter}  hold for integers up to $n-1$. The case $n=1$ of Theorem \ref{Lipcenter} is already proved so we may assume that $n>1$.

By Theorem \ref{Lipcenter} in the case $n=1$ that we just proved and then by logical compactness, it is enough to consider the case of an $\cL$-definable part $X\subset A$ such that, for some coordinate projection $p:K^n\to K^{n-1}$ and some $\cL$-definable functions $c:p(X)\to K$  and $d:p(X)\to K$, $X_{w}$ is in $c(w)$-config, and $f_w(X_w)$ is in $d(w)$-config for each $w\in p(X)$. Again by Theorem \ref{Lipcenter} in the case $n=1$ and logical compactness (we keep for simplicity the notation $X$ for the part of $A$ we have now to work on), we may assume that, for some coordinate projection $p_1:p(X)\to K^{n-2}$, there are $\cL$-definable functions $b,c',d':p_1(p(X))\to K$ such that $p(X)_v$ is in $b(v)$-config, $c_v(p(X)_v)$ is in $c'(v)$-config, and $d_v(p(X)_v)$ is in $d'(v)$-config, for each $v\in p_1(p(X))$.
By Corollary \ref{cor:invers} and by Theorems \ref{Lip} and \ref{Lipcenter}  for $n-1$, we may suppose that $X$ is open.
The reduction to the case where $c$ is $1$-Lipschitz continuous is done as in \cite{CCL}, using s-continuity instead of the norm of the partial derivatives. Let us now describe the strategy of \cite{CCL} to make $c$ $1$-Lipschitz continuous, where we refer to \cite{CCL}, proof of Proposition 2.4, for the explicit ultrametric calculations. We shall proceed by decreasing
induction on the number of variables on which $c$ depends nontrivially, 
the case when is no such variable being clear. By compactness and tameness we may assume that $c$ is $s$-continuous in each variable separately. After reordering the variables $x_1,\ldots,x_{n-1}$, we may suppose that $|\partial c/\partial x_{n-1}|$ is maximal among the $|\partial c/\partial x_{i}| $ on the whole of $p(X)$ for $i=1,\ldots,n-1$. If $|\partial c/\partial x_{n-1}|\leq 1$ on the whole of $p(X)$, then we are done by Theorem \ref{Lip} for $n-1$. Hence, we may further assume that $1< |\partial c/\partial x_{n-1}|$ on the whole of $p(X)$. Now we subdivide in two cases (possibly involving a further finite partitioning), where for the complete details we refer to the two cases in the proof of Proposition 2.4 of \cite{CCL}: either also $X_w$ is in $c'(v)$-config for each $v$ and each $w$ with $p_1(w)=v$ in which case we are done by induction on the number of variables on which $c$ depends nontrivially, or, the graph of $c$ is included in $X$. In the latter case, one can finish by taking the inverse function of $c_v$ and by reversing the role of $x_n$ and $x_{n-1}$, using that $p(X)_v$ is in $b(v)$-config, and that $c_v$ is s-continuous and injective, as in \cite{CCL}. In this case one concludes similarly as in \cite{CCL}, using the chain rule as in the proof of Corollary \ref{cor:invers}. Thus we may suppose that $c$ is $1$-Lipschitz.

Let us now show that we can reduce further to the case where $d$ is $1$-Lipschitz as well, as required by the theorem. The argument is by decreasing induction on the number of variables on which $d$ depends.

Let us summarize the relevant current assumptions.
The set $X_{w}$ is in $c(w)$-config, $f_w(X_w)$ is in $d(w)$-config, $d_v(p(X)_v)$ is in $d'(v)$-config, and $f_w$ and $d_v$ are s-continuous for each $w\in p(X)$ and for each $v\in p_1(p(X))$. 
Furthermore, we may assume that $c$ is $1$-Lipschitz, $X$ is open and $f_w$ 
is injective for each $w\in p(X)$.
Moreover, by Theorem \ref{Lip} for $n-1$ and by compactness we may require that $f(\cdot,x_n)$ is $1$-Lipschitz for each $x_n$.
Define $Y$ as the image of $X$ under the function
$X\to K^n$ sending  $x$ to $(x_1,\ldots,x_{n-1},f(x))$. We may suppose that there are $\xi_1,\xi_2\in k$ such that, for each $w\in p(X)$, one has
$$
X_{w} = \{x_n\in K\mid \ac(x_n -c(w) )=\xi_1,\ |x_n -c(w) | \in G_1(w)\} \hbox{ and }
$$
$$
Y_{w}  = f_w(X_w)  =\{z\in K\mid \ac(z - d(w))=\xi_2,\ |z-d(w)| \in G_2(w)\}
$$
for some sets $G_i(w) \subset \Gamma$.
By Proposition \ref{continuity} for $n=1$ and by compactness, we may suppose that for each $v\in p_1(p(X))$ and each open ball $B$ contained in $p(X)_v$, these sets $G_i(v,t)$ do not depend on the choice of $t\in B$.
\\

\textbf{Case 1.} The function $w\mapsto f(w, z+c(w))$ is $1$-Lipschitz continuous for each $z$, where $z$ and $w$ are such that $(w, z +c(w))\in X$. \\

In this case, we may perform the bi-$1$-Lipschitz transformation $(w,z)\mapsto (w,z+c(w))$, and assume that $c$ is identically zero. This transformation preserves the assumptions summarized above. If $d$ is locally $1$-Lipschitz continuous in each variable separately, we are done by Proposition \ref{continuity} and by Theorem \ref{Lip} for $n-1$. Hence, we may suppose that $d$ is not locally $1$-Lipschitz in at least one variable. By working piecewise, we may suppose $d$ is nowhere locally $1$-Lipschitz in at least one specific variable.
Up to reordering the variables $x_1,\ldots,x_{n-1}$ if necessary, we may thus suppose for any $v$ that $d_v$ is nowhere locally $1$-Lipschitz.

Suppose that there is $v\in p_1(p(X))$ and an open ball $B$ contained in $p(X)_v$ such that
\begin{equation}\label{t1t2}
Y_{v,t_1}
 \not = Y_{v,t_2}
\end{equation}
for some $t_1,t_2\in B$.
Then this violates the $1$-Lipschitz continuity of $f$ in the variable $x_{n-1}$ as follows. Fix
$t_1,t_2\in B$ satisfying (\ref{t1t2}).
Choose $\gamma_0$ in $G_2(v,t_1)=G_2(v,t_2)$ such that the sets $A_1$ and $A_2$ are disjoint balls, with
$$
A_i := \{y\in K \mid \ac(y - d(v,t_i))=\xi_2,\ |y-d(v,t_i)|  = \gamma_0 \}.
$$
By Remark \ref{rem:comp} on the correspondence of maximal balls in domain and range of the functions $f_{v,t_i}$, we can take $\gamma$ in $G_1(v,t_1) = G_1(v,t_2)$,
and $x_{n}$ with
$$
\ac(x_{n} )=\xi_1,\ |x_{n} | =\gamma
$$
such that $f(v,t_i,x_n)$ lies in $A_i$ for $i=1,2$.
By  s-continuity of $d_v$ and the fact that $d_v$ is nowhere locally $1$-Lipschitz continuous, and the note below Definition \ref{defjacprop}, one has
$$
| d_v(t_1) - d_v(t_2)  | > |t_1-t_2|
$$
Since $A_1$ and $A_2$ are disjoint, it follows from their description that one has for any $y_i\in A_i$ for $i=1,2$ that
$$
|y_1-y_2| \geq |d(v,t_1)-d(v,t_2)|.
$$
Combining these inequalities with $y_i= f(v,t_i,x_n)$, one finds
$$
| f(v,t_1,x_n) - f(v,t_2,x_n)| \geq | d_v(t_1) - d_v(t_2)  | > |t_1-t_2|
$$
which indeed violates the $1$-Lipschitz continuity of $f$ in the variable $x_{n-1}$.
Hence, we may suppose that for each $v\in K^{n-2}$ and each open ball $B$ contained in $p(X)_v$, the set $Y_{v,t}$ is independent of the choice of $t\in B$. But then it follows that $Y_{v,t}$ is in $d'(v)$-config, as we wanted to prove.
\\

\textbf{Case 2.} The graph of $c$ is contained in $X$ and $f(w,c(w))=d(w)$. \\

In this case we may assume, by Theorem \ref{Lip} for $n-1$ and compactness, that $w\mapsto f(w,c(w))$ is $1$-Lipschitz as well. But then it follows from the case assumption, namely from $f(w,c(w))=d(w)$ and from a chain rule, see below Definition \ref{defjacprop}, that $d$ is $1$-Lipschitz and we are done also for this case.
\\

We now explain how one can deduce the general case from Case 1 and Case 2.
Let us write
$$
X^{(0)} = \{(w,z)\in p(X)\times K\mid (w,z + c(w) )\in X \},
$$
so that
$$
X^{(0)} = \{(w,z)\in p(X)\times K\mid \ac(z )=\xi_1,\ |z| \in G_1(w)\}
$$
and $f^{(0)}:X^{(0)}\to K$ for the function sending $(w,z)$ to $f(w,z + c(w) )$.

Take an $\cL$-definable function
$$
g:X^{(0)} \to k^s 
$$
for some $s$ such that, for each $\eta\in k^s$, each $(w,z)\in X^{(0)}$ the function $f^{(0)}(\cdot,z)$ is $1$-Lipschitz on $K^{n-1}\times \{z\} \cap g^{-1}(\eta)$ for each $z$.
In general  $g^{-1}(\eta)_w$ may not be in $0$-config. However, we may assume that  $g^{-1}(\eta)_w$ is in $\tilde c(\eta,w)$-config for some $\cL$-definable function $\tilde c$.
We may also assume that either the graph of $\tilde c$ is disjoint from $X^{(0)}$, or is included in $X^{(0)}$. In the first case, one notes that $g^{-1}(\eta)_w$ is in fact in $0$-config, and one falls in Case 1.
In the remaining case when the graph of $\tilde c$ is included in $X^{(0)}$, we may suppose that $d(w) = f(w,\tilde c(w))$ by the proof of Theorem \ref{Lipcenter} for $n=1$. If $\tilde c$ is $1$-Lipschitz then we fall in Case 2 and we are done.
If $\tilde c$ is not $1$-Lipschitz, after performing a permutation of  the variables as we did
in the beginning of this proof for $c$,  and transforming $d$ accordingly, we may assume $\tilde c$ is $1$-Lipschitz, and one falls again in a case already treated.
\end{proof}

\begin{remark}
We amend on \cite{CCL} and \cite{ch}, more precisely on their proofs of the piecewise Lipschitz continuity results.
The explanation of the reduction to the Cases 1 and 2 as in the proof of Theorem \ref{Lipcenter} is not given in the proof of Theorem 2.3 of  \cite{CCL}, and, only Case 1 is treated in \cite{CCL}, namely by assuming (*) on page 83 of \cite{CCL}.
Either one adds a Case 2 and a reduction to Cases 1 and 2, or, one uses the simplified approach of this paper. If one uses the approach of the present paper, one should adapt \cite{ch} accordingly, and use the monomial approximation result of \cite{ch} to get rid of the constant $|1/N|$ as created in the proof of Theorem \ref{Lipmixed} below for the analogues of (\ref{equi000}) and (\ref{equi00}).
\end{remark}

\begin{proof}[Proof of Theorem \ref{Lip}] 



We proceed by induction on $n$, assuming that Theorem \ref{Lipcenter} holds for integers up to $n$. For $n=0$ there is nothing to prove concerning
the statement of Theorem \ref{Lip}.
Write $p:X\to K^{n-1}$ for the coordinate projection sending $x:=(x_1,\ldots,x_n)$ to $\hat x := (x_1,\ldots,x_{n-1})$ and
define $Y$ as the image of $X$ under the function
$X\to K^n$ sending  $x$ to $(\hat x,f(x))$.

Clearly, by induction on the number of variables on which $f$ depends, Lemma \ref{invers}, Corollay \ref{cor:invers}, Theorem \ref{Lipcenter}, tameness, compactness and by going to a test pair, we may assume that the following basic assumptions hold.

\textbf{Basic Assumptions.} 
\begin{enumerate}
\item[(0)] $X$ is open in $K^n$.
\item[(1)] $f(\hat x,\cdot)$ is s-continuous for each $\hat x$ in $p(X)$.
\item[(2)] $f(\cdot,x_n)$ is $1$-Lipschitz continuous for each $x_n$.
\item[(3)] for each $\hat x$ in $p(X)$, the set $X_{\hat x}$ is in $c(\hat x)$-config, where $c$ is an $\cL$-definable function.
\item[(4)] for each $\hat x$ in $p(X)$, the set $Y_{\hat x}$ is in $d(\hat x)$-config, where $d$ is an $\cL$-definable function.
\item[(5)] $c$ and $d$ are $1$-Lipschitz on $p(X)$. 
\end{enumerate}

We prove that under these basic assumptions, $f$ is globally $1$-Lipschitz. By replacing $f$ by $f-d$ we may suppose that $d=0$. We may also assume that
\begin{equation}\label{equi0000}
|x_n-c(\hat x)|\leq |x_n| \mbox{ for each $x\in X$.}
\end{equation}
Indeed, one can replace $c$ by $0$ on the piece where one has $|x_n-c(\hat x)| > |x_n|$.

Consider $x,y\in X$. If $x_n$ and $y_n$ lie in the same open ball $B$ which is included in $X_{\hat x}$ with $\hat x = p(x)$, then one derives from the  assumptions:
\begin{eqnarray*}
 |f(x) - f(y)|
 & = & |f(x) - f(\hat x,y_n) + f(\hat x,y_n) - f(y)| \\
 & \leq & \max(|f(x) - f(\hat x,y_n)|,\  |f(\hat x,y_n) - f(y)| ) \\
 & \leq & \max(|x_n - y_n|,\  |\hat x - \hat y| ) \\
 & = & | x- y |,
   \end{eqnarray*}
which ends the proof in this case.

Now suppose that $x_n$ and $y_n$ do not lie in any open ball  included in $X_{\hat x}$ with $\hat x = p(x)$, and, by symmetry, that $x_n$ and $y_n$ do not lie in any open ball which is included in $X_{\hat y}$ with $\hat y = p(y)$. Note that this implies that
\begin{equation}\label{equi000}
 |x_n-c(\hat x)| \leq |x_n - y_n| \mbox{ and }  |y_n-c(\hat y)|   \leq |x_n - y_n| .
\end{equation}

We also have
\begin{equation}\label{equi00}
|f(x) | \leq   |x_n-c(\hat x)| , \mbox{ and }   |f(y) | \leq   |y_n - c (\hat y) |,
\end{equation}
by s-continuity as given by (1), since $f$ is locally $1$-Lipschitz and $d=0$.
Combining (\ref{equi0000}), (\ref{equi000}), (\ref{equi00}) one gets
$$
|f(x) - f(y)|  \leq \max (|x_n-c(\hat x)|, |y_n - c (\hat y) | ) \leq | x_n- y_n | \leq | x- y |,
$$
and we are done.
\end{proof}

\subsection{Lipschitz continuity in mixed characteristic tame geometry}\label{sec:mix}

Recall that the generalized Denef-Pas language $\LPas$ consists of the sorts $\mathrm{VF}$ for valued field, the $R_n$ for $n\geq 1$ for the residue rings modulo the product ideal of the ideal $(n)$ and the maximal ideal, and $\VG$ for the union of $\{0\}$ with the  multiplicatively written value group $\VG^\times$, and having as symbols the ring language on $\mathrm{VF}$, the ring language on the $R_n$, the language of ordered multiplicative groups $(\cdot,<)$ on $\VG^\times$, the norm map from $\mathrm{VF}$ to $\VG$, and angular component maps $\ac_n:\mathrm{VF}\to R_n$ for all $n\geq 1$. An angular component map $\ac_n:\mathrm{VF}\to R_n$ is just a multiplicative map $\ac_n:\mathrm{VF}^\times \to R_n^\times$, extended by zero on zero, that coincides with the natural projection on the units of the valuation ring to $R_n$. The maps $\ac_n$ are required to form a compatible system, that is, the composition of the projection $R_n\to R_m$ with $\ac_n$ has to equal $\ac_m$ whenever $m$ divides $n$.

Let $\cL_+$ be this generalized language of Denef-Pas $\LPas$, but without multiplication on $\VF$ and on the $R_n$.
Let $\cL$ be any first order language with the same sorts as $\cL_{+}$ and such that $\cL$ contains all the symbols of $\cL_{+}$.
Let $K$ be an $\cL_+$ structure, where we write $\Gamma$ for $\mathrm{VG}(K)$, $\Gamma^\times$ for $\VG^\times(K)$ and $K_n$ for $R_n(K)$.
Call a set $S$ auxiliary if it is a subset of a Cartesian product of some copies of $\Gamma$ and the $K_n$.

Let us use the notation
$$
A_{\xi,\gamma} := \{t\in K\mid \ac_n(t)=\xi,\ |t|=\gamma\},\ \mbox{ for
$\xi\in K_n$ and $\gamma\in \Gamma$.}
$$
Furthermore, let $(K_s)^s$ when $s=0$ be shorthand for $\{0\}$.

\begin{defn}[configurations]\label{config2}
Given a set $T\subset K\times S$ with $S$ auxiliary,
say that $T$ is in $c$-config with depth $n$ if there exists $\xi\in K_n$ such that
$T$ equals the union over $\gamma\in \Gamma$ of sets $(c + A_{\xi,\gamma})\times U_{\gamma}$ for some $U_{\gamma}\subset S$.
Again, if $T$ is nonempty and in $c$-config, then $n$, $\xi$ and the $U_\gamma$, where $\gamma$ is such that $\emptyset\not=A_{\xi,\gamma}$, are uniquely determined by $T$ and $c$.

Say that $T$ is in $\cL$-tame config if there exist $n\geq 0$, $s \geq 0$ and $\cL$-definable functions
$$
g:K\to K_s^s\ \mbox{ and }\
c:K_s^s\to K
$$
such that the range of $c$ contains no open ball, and such that
$T\cap (g^{-1}(\eta) \times S )$ is in $c(\eta)$-config with depth $n$ for each $\eta\in K_s^s$.
\end{defn}

One could as well have used $(K_s)^t$ instead of $K_s^s$ for some independent $s$ and $t$, but this would create heavier notation.

Suppose that the $\cL_+$ structure on $K$ can be expanded to be a structure for the generalized Denef-Pas language. This condition is a simplification that replaces some of the conditions of Section \ref{sec:5b}, which would otherwise have become more cumbersome in the mixed case. We maintain this condition throughout this section.
Suppose moreover that $K$ is an $\cL$-structure.

Consider an $\cL$-structure $M$ which is elementarily equivalent to $(K,\cL)$. 
If $L$ is a language which is obtained from $\cL$ by adding constants from $M$, then we call $(M,L)$ a test pair for $(K,\cL)$.

\begin{defn}[Mixed tameness]\label{mixKtame}
Say that $(K,\cL)$ is weakly mixed tame if the following hold:
\begin{enumerate}
\item[(1)] each $\cL$-definable set $T\subset K\times S$ with $S$ an auxiliary set is in $\cL$-tame config;
\item[(2)]  if $F:X\subset K\to K$ is $\cL$-definable, then there exists an $\cL$-definable function
$g:X\to K_s^s$ for some $s\geq 0$ such that, for each $\eta\in K_s^s$, the restriction of $F$ to $g^{-1}(\eta)$ is s-continuous.
\end{enumerate}


Say that $(K,\cL)$ is mixed tame if each test pair $(M,L)$ for $(K,\cL)$ is weakly tame.

More generally, call an $\cL$-theory $\cT$ mixed tame if for each model $\cM$ of $\cT$, the pair $(\cM,\cL)$ is mixed tame.
\end{defn}

By essentially the same proof as that of Theorem \ref{Lip}, we obtain our final result on Lipschitz continuity. 

\begin{thm}\label{Lipmixed}
Suppose that $(K,\cL)$ is mixed tame.
Let $f:X\subset K^n\to K$ be an $\cL$-definable function which is locally $1$-Lipschitz. Then there exists an integer $N>0$ and an $\cL$-definable function
$$
g:X\to K_s^s
$$
for some $s\geq 0$ such that for each $\eta\in K_s^s$, the restriction of $f$ to $g^{-1}(\eta)$ is $|1/N|$-Lipschitz.
\end{thm}
\begin{proof}
Adapt the proof of Theorem \ref{Lip}, and all its auxiliary results and their proofs, by replacing any occurrence of (the residue field) $k$ by a residue ring $K_s$ for some $s$, and each occurrence of $\ac$ by $\ac_s$ for some $s$. (See for example Proposition \ref{inj-cons-mixed} for the adaptation of Proposition \ref{inj-cons}.) In this process, a constant of the form $|1/N|$ shows up
in the upper bound in equations (\ref{equi000}) and (\ref{equi00}), where $N$ can be bounded in terms of the depths of the occurring configurations.
\end{proof}

\begin{remark}The natural analogue with $|1/N|$-Lipschitz centers for some integer $N>0$ of Theorem \ref{Lipcenter} in mixed characteristic also holds. We leave its proof to the reader. Instead, we make explicit the analogue of Proposition \ref{inj-cons}, as Proposition \ref{inj-cons-mixed}.
\end{remark}

\begin{prop}[Injectivity versus constancy
]\label{inj-cons-mixed}
Suppose that $(K,\cL)$ is mixed tame.
Let $F:X\subset K\to K$ be $\cL$-definable. Then there exist an integer $s \geq 0$ and an $\cL$-definable function
$$
g:X\to K_s^s
$$
such that for each $\eta\in K_s^s$ the restriction of $F$ to $g^{-1}(\eta)$ is injective or constant.
\end{prop}
\begin{proof}
Similar adaptation of the proof of Proposition \ref{inj-cons} as explained in the proof of Theorem \ref{Lipmixed}.
\end{proof}

\subsection{Examples and some corollaries}\label{ex-tame}




The following proposition provides examples of (mixed) tame structures.

\begin{prop}[\cite{CLip}, Theorem 6.3.7]\label{pex}
Let $\LPas$ be the generalized Denef-Pas language.
Suppose that $K$ is a valued field of characteristic zero, equipped with angular component maps $\ac_n$, and, a separated analytic $\cA$-structure as in Definition 4.1.6 of \cite{CLip}, where $\cA$ is a Weierstrass system as in Definition 4.1.5 of \cite{CLip}, and write $\cL$ to denote the corresponding expansion of $\LPas$. Then $(K,\cL)$ is tame, resp.~mixed tame, if $K$ is of equicharacteristic zero, resp.~of mixed characteristic.
\end{prop}
\begin{proof}
One readily derives this statement from the version of Theorem 6.3.7 of \cite{CLip} which is formulated with sorts for quotients $K^\times / 1+ n\cM_K$ with $\cM_K$ the maximal ideal of the valuation ring of $K$, instead of with the sorts $R_n$. Note that $K$ is automatically henselian because it has a separated analytic $\cA$-structure with $\cA$ a Weierstrass system.
\end{proof}

In some specific cases Proposition \ref{pex} follows, alternatively, by results from \cite{D2}, \cite{Pas1}, \cite{Pas2}, \cite{Ccell}, resp.~\cite{CLR}; see also \cite{Mourgues} for a related, one-sorted result in the $p$-adic subanalytic case.

All the structures that we will use in Section \ref{sec:2} are tame or mixed tame, by Examples 4.4(1) and 4.4(13) of \cite{CLip}, Section 3.4 of \cite{CLip2}, and by Proposition \ref{pex}.

Note that by compactness, family versions of Theorems \ref{Lip}, \ref{Lipmixed} and  \ref{Lipcenter} follow naturally. Uniform versions in all models of a tame theory follow likewise.

\section{Non-archimedean Yomdin-Gromov parametrizations with Taylor approximation}\label{sec:2}\label{sec:GYT}

\subsection{}

In Section \ref{sec:5}, piecewise Lipschitz continuity was obtained for a definable function with bounded first partial derivatives. In this section, we will show that one can parametrize any definable set by a small set of maps  with bounded partial derivatives up to any given finite order.
In the previous sentence, a small set of maps means a set of maps indexed definably by the residue field, or more generally, some residue rings $L_N$ for some $N>0$.
We will use piecewise Lipschitz continuity from the previous section, together with new techniques using a strong kind of analyticity. We will
furthermore define a property $T_r$ for approximation by Taylor polynomial of degree $r-1$ with remainder term of degree $r$ and the property $T_r$ will be required in our parametrizations (see Theorem \ref{GYT}). To distinguish from the more abstract setting of Section \ref{sec:5}, we will write $L$, instead of $K$, in this section, where  $L$ will be a
valued field.

\subsection{}\label{subsec:GYT}

Let $L$ be a complete, discretely valued field of characteristic zero such that, for each integer $n>0$ the set $P_n(L)$ of the $n$-th powers in $L^\times$ has finite index in $L^\times$.
Write $\cO_L$ for the valuation ring of $L$ with maximal ideal $\cM_L$, and
residue field $k_L$. Let us choose a uniformizer $\varpi_L$ of $\cO_L$ and let us write $p_L\geq 0$ for the characteristic of $k_L$.
Write $\ord:L^\times \to \ZZ$ for the (surjective) valuation map and write $|\cdot|$ for the multiplicative norm on $L$ with normalization $|\varpi_L|=e_L$ for some real number $e_L<1$.
When $k_L$ is finite  we set $e_L = \vert k_L \vert^{-1}$. As usual, the norm $|x|$ of a tuple $x=(x_1,\ldots,x_n)$ is set to be the maximum of $|x_i|$ for $i=1,\ldots,n$, and, in the case where $L$ is a $p$-adic field, $e_L$ is taken to be the inverse of the number of elements in $k_L$.

Further, for each $n\geq 1$ write $L_n$ for $\cO_L\bmod n\cM_L$ and write $\ac_n : L \to L_n$ for the function which sends $0$ to $0$ and nonzero $x\in L$ to $x\varpi_L^{-\ord x} \bmod n\cM_L$. Note that, if $k_L$ has characteristic $0$, then one has $L_n=L_1$ and $\ac_n=\ac_1$ for all $n\geq 1$. We also write $\ac$ for $\ac_1$.

Let $\LPas^L$ be the first order language having sorts for $L$, the $L_n$ for $n\geq 1$, and $\ZZ$, and having as symbols the ring language with parameters from $L$ on $L$, the ring language on $k_L$, the Presburger language $(0,1+,-,\leq,\{\equiv_n\}_{n>1})$ on $\ZZ$, the valuation map $\ord:L^\times \to \ZZ$, and the maps $\ac_n:L\to L_n$ for all $n\geq 1$.

A function
$$
f: L^n\to L 
$$
which satisfies $f(x)=0$ whenever $x\in L^n\setminus \cO_L^n$ is called a restricted analytic function if there is a power series $\sum_{i\in\NN^n} a_iX^i$ in $\cO_L\llb X_1,\ldots,X_n \rrb$, converging on $\cO_L^n$, such that $f(x) = \sum_{i\in\NN^n} a_ix^i$ for $x\in\cO_L^n$.
Let $\Lan^L$ be the language consisting of $\LPas^L$ and all the restricted analytic functions $L^n\to L$ for all $n\geq 0$.
An $\LPas^L$-definable subset of $L^n$ is often called a semialgebraic subset of $L^n$.
An $\Lan^L$-definable subset of $L^n$ is called a subanalytic subset of $L^n$.
Let $\cL$ be either $\LPas^L$ or $\Lan^L$. From now on in Section \ref{sec:2}, definable sets and functions will be so for the language $\cL$. Note that the study of definable sets was initiated in the works of Macintyre \cite{Mac}, Denef and van den Dries \cite{DvdD} in the $p$-adic case, and was generalized later to this and other settings in e.g. \cite{CLR}, \cite{BelairMS}, \cite{CLip}, and \cite{Rid}.

For a nonempty definable set $X\subset L^n$, the dimension of $X$ is defined as the largest integer $m\leq n$ such that, for at least one of the coordinate projections $p:L^n\to L^m$, the set $p(X)$ has nonempty interior for the valuation topology on $L^m$. The empty set is given dimension $-\infty$.

For an integer $r\geq 0$, and similarly for $r=+\infty$, the $C^r$-norm of a $C^r$-function $f=(f_1,\ldots,f_n):U\to L^n$ on an open $U\subset L^m$ is defined as the supremum over all $x$ in $U$, all $i=1,\ldots,n$, and all $\alpha \in \mathbb{N}^m$ with $\vert \alpha \vert \leq r$, of the values
\begin{equation}\label{Cr}
\Bigl \vert \frac{1}{ \alpha !} \frac{\partial^{\alpha} f_i}{\partial x^{\alpha}} (x) \Bigr \vert, 
\end{equation}
where $\alpha!$ stands for $\prod_{j=1}^m (\alpha_j!)$, and $\vert \alpha \vert $ for $\sum_j \alpha_j$.

For a $C^r$-function $(f_i)_i=f:U\subset L^m\to L^n$ and $y\in U$, write $T^{<r}_{y, f}$ (or $T^{\le r-1}_{y, f}$) for the tuple of the Taylor polynomials of the $f_i$ at $y$ of degree $r-1$. We now define a notion of (global) approximation by Taylor polynomials.

\begin{defn}\label{Tr} Let $X$ be a subset of $L^m$. Let $r$ be a positive integer.
We say that a map
$f = (f_1, \cdots, f_n) : X \to  L^n$ satisfies
$T_r $ (on $X$) if $X$ is open in $L^m$, $f$ is $C^r$ with $C^r$-norm not larger than $1$, and
for every $x$ and $y$ in $X$ one has
$$
\vert f (x) - T^{< r}_{y, f} (x) \vert \leq \vert x - y \vert^{r}.
$$
\end{defn}

If a map $f$ satisfies $T_r$ on a subset of  $\cal O_L^m$, then it also satisfies $T_{\ell}$ for any $\ell$ with $1\leq \ell \leq r$, by  the ultrametric inequality. For a $C^1$-function from an open subset of $L^m$ to $\cO_L^n$, satisfying $T_1$ is equivalent to being Lipschitz continuous with Lipschitz constant $1$.

\begin{defn}
Let $X\subset L^n$ be a definable set of dimension $m$.
A family $f_i:P_i\to X$ of definable functions for $i$ running over some set $I$ and with $P_i\subset \cO_L^m$, is called a $T_r$-parametrization of $X$ if each of the $f_i$ satisfies $T_r$ and
$$
X = \bigcup_{i\in I} f_i(P_i).
$$
\end{defn}

A family $X$ of sets $X_y$ for $y$ running over a definable set $Y$ is called a definable family if $\{(x,y)\mid x\in X_y,\ y\in Y\}$ is a definable set. A collection of maps $f_y$ for $y$ running over a definable set $Y$ is called a definable family of maps if the collection $f$ of the graphs of the $f_y$ is a definable family of sets.
Recall that for $f:D\subset A\times B\to C$ a map and $a\in A$, we write $D_a$ for the set $\{b\in B\mid (a,b)\in D\}$ and $f(a,\cdot)$ or $f_a$ for the function $b\mapsto f(a,b)$ on $D_a$.

\begin{thm}\label{GYT}
Let $n\geq 0$, $m\geq 0$ and $r >0$ be integers and let $(X_y)_{y\in Y}$ be a definable family of subsets $X_y \subset \cO_L^n$ for $y$ running over a definable set $Y$. Suppose that $X_y$ has dimension $m$ for each $y\in Y$.
Then there exist an integer $N$ and a definable family $g=(g_{y,i})_{y\in Y, i\in L_N^N}$ of definable functions
$$
g_{y,i}:P_{y,i}   \to X_y 
$$
such that $P_{y,i} \subset \cO_L^{m}$ and for each $y$, $(g_{y,i})_{i\in L_N^N}$ forms a $T_r$-parametrization of $X_y$.
Namely,
$$
X_y = \bigcup_{i\in L_N^N} g_{y,i}(P_{y,i}) \mbox{ for each } y\in Y,
$$
and $g_{y,i}$ satisfies $T_r$ on $P_{y,i}$ for each $y\in Y$ and each $i\in L_N^N$.
\end{thm}

Let us first describe the strategy of the proof of Theorem \ref{GYT}. We introduce a global notion of analyticity (global in the sense that the radii of convergence of the power series are large in a certain sense), and show a globally analytic cell decomposition theorem. A first step
towards Theorem \ref{GYT} is to parametrize our set with functions having small $C^1$-norm. This is   done  by inverting the roles of some of the coordinates and using the chain rule to bound the $C^1$-norm by  $1$.
The analyticity allows us to go further by working with Gauss-norms on balls and on boxes (defined as products of balls). This has two uses: to obtain that composition with well chosen power maps makes the $C^r$-norm less or equal to $1$, and to obtain that $C^r$-norm bounded by $1$ implies $T_r$ on each maximal ball included in the domain. Moreover, by using the results of the previous section on $T_1$, we reduce to the situation where one has globally $T_1$ and locally (on maximal balls) $T_r$. By composing with power maps once more, the previous $T_1$ condition on far away points implies $T_r$, which then  follows globally on the pieces. Finally, Theorem \ref{GYTs} strengthens Theorem \ref{GYT}.

\begin{defn}[Cell around zero] Consider integers $n\geq 0$ and $n_i>0$ for $i=1,\ldots, n$.
A nonempty definable set $X\subset L^n$ is called a cell around zero with depth $(n_i)_{i\in \{1,\cdots,n\}}$ if it is of the form
$$
\{x\in L^n\mid \ac_{n_i}(x_i)=\xi_i,\, (\vert x_1|, \cdots,\vert x_n \vert)\in G \},
$$
for some set $G\subset \RR^n$ and some $\xi_i\in L_{n_i}$. If moreover $G$ is a subset of $(\RR^\times)^n$, where $(\RR^\times)^0=\{0\}$ and $L^0=\{0\}$ by convention, then $X$ is called an open cell around zero.
More generally, for nonempty definable sets $Y$ and $X\subset Y\times L^n$, the set $X$ is called a cell around zero over $Y$ with depth $(n_i)_{i\in \{1,\cdots,n\}}$ if it is of the form
$$
\{(y,x) \in Y\times L^n\mid y\in Y,\ \ac_{n_i}(x_i)=\xi_{i}(y),\, (y,(|x_i|)_i)\in G  \},
$$
for some set $G\subset Y\times \RR^n$ and some definable functions $\xi_i: Y\to L_{n_i}$. If moreover $G$ is a subset of $Y\times (\RR^\times)^n$, then $X$ is called an open cell around zero over $Y$. Note that the definability of $G$ is not an issue here since $X$ is assumed to be definable.
\end{defn}

Note that some of the sets $X_y$ for some $y\in Y$ may be empty.

By a box we mean a Cartesian product of closed balls, where a closed ball is a subset of $L$ of the form
$$
\{x\in L\mid |x-c| \leq |r|\}
$$
for some $r$ in $L^\times$ and some $c\in L$. For a box $B$ of the form $\prod_{i=1}^n\{x\in L\mid |x-c_i| \leq |r_i|\}$, we define the associated set $B_{\rm as}$ as
\begin{equation}\label{Bass}
B_{\rm as}:=\prod_{i=1}^n\{x\in L^{\rm alg}\mid |x-c_i| < |r_i/\varpi_L|\},
\end{equation}
where $L^{\rm alg}$ is an algebraic closure of $L$ with norm extending the one on $L$. Note that $B\subset B_{\rm as}$.
We extend the definition of $C^r$-norm, for an integer $r\geq 0$ and also for $r=+\infty$, of a $C^r$-function $f=(f_1,\ldots,f_n):U\to (L^{\rm alg})^n$ on an open $U\subset (L^{\rm alg})^m$ in the obvious way.

\begin{defn}[Global analyticity]\label{glob}
Let $f:X\subset L^m\to L^n$ be a definable function on an open set $X$.
Say that $f$ is globally analytic on $X$ if, for any box $B$ contained in $X$, the restriction of $f$ to $B$ is given by a tuple of power series, converging on the associated set $B_{\rm as}$. By this we mean that for any $b\in B$, there is a power series $\sum a_i x^i $ such that $f(x) = \sum a_i (x-b)^i $ for all $x\in B$ and which converges on $-b +B_{\rm as}$.
\end{defn}

The notion of globally analytic maps will be most useful when the domain is an open cell around zero.
We can now give in Theorem \ref{GYTs} a strengthened version of Theorem \ref{GYT} as well as a variant of Theorem \ref{GYTs} in Proposition \ref{gyts}.

\begin{thm}\label{GYTs}
With data and notation from Theorem \ref{GYT}, one can take $g$ as in Theorem \ref{GYT} and such that moreover $g_{y,i}$ is globally analytic for each $y,i$ and such that $P$ is an open cell around zero over $Y\times L_N^N$, where $P=\{(y,i,x)\mid y\in Y,\ i\in L_N^N,\ x\in P_{y,i}\}$.
\end{thm}


\begin{prop}\label{gyts}
With data and notation from Theorem \ref{GYT}, one can take $g$ as in Theorem \ref{GYTs} such that moreover 
for each $y,i$, and for each box $B$ contained in $P_{y,i}$ and associated set $B_{\rm as}$, the power series corresponding to $g_{y,i}$ on $B$ satisfies $T_r$ on the whole of $B_{\rm as}$.
\end{prop}



\subsection{}

In order to prove Theorems \ref{GYT} and \ref{GYTs}, we give now some preliminary definitions.

For a closed ball
$$
\{x\in L\mid |x-c| \leq |r|\}
$$
with $r$ in $L^\times$ and $c\in L$, the real number $|r|$ is called the 
radius of the closed ball, while $\ord r$ is called the valuative radius. All balls from now on will be closed balls, as opposed to Section \ref{sec:5} where we  used open balls. A ball $B$ with $B\subset X$ for some set $X\subset L$ is called a maximal ball contained in $X$ if $B$ is a closed ball which is maximal for the inclusion among all closed balls contained in $X$. By convention, $L^0$ stands for $\{0\}$, and so do also $k_L^0$ and $\ZZ^0$, namely the definable set of a true formula without free variables. 

We complement the above notion of cells around zero by a notion of cells with a center.

\begin{defn}[Cell with center]\label{cellc}
Consider integers $n\geq 0$ and $n_i>0$ for $i=1,\ldots, n$.
For non empty definable sets $Y$ and $X\subset Y\times L^n$, the set $X$ is called a cell over $Y$ with center $(c_i)_{i=1, \cdots, n}$ and  depth $(n_i)_{i=1, \cdots, n}$ if it is of the form
$$
\{(y,x) \in Y\times L^n\mid y\in Y,\ \ac_{n_i}(x_i-c_i(x_{<i}))=\xi_{i}(y),\, (y,(|x_i-c_i(x_{<i})|)_i)\in G  \},
$$
for some set $G\subset Y\times \RR^n$ and some definable functions $\xi_i: Y\to L_{n_i}$ and $c_i:Y\times L^{i-1}\to L$, where $x_{<i}=(y,x_1,\ldots,x_{i-1})$. If moreover $G$ is a subset of $Y\times (\RR^\times)^n$, where $(\RR^\times)^0=\{0\}$, then $X$ is called an open cell over $Y$
(with center $(c_i)_{i=1, \cdots, n}$ and  depth $(n_i)_{i=1, \cdots, n}$).
\end{defn}

\begin{defn}[Associated cell around zero]\label{cell0}
Let $X$ be a cell over $Y$ with center, with notation from Definition \ref{cellc}. The cell around zero associated to $X$ is by definition the cell $X^{(0)}$ obtained by forgetting the centers, namely
$$
X^{(0)}=\{(y,x) \in Y\times L^n\mid y\in Y,\ \ac_{n_i}(x_i)=\xi_{i}(y),\, (y,(|x_i|)_i)\in G  \}
$$
with associated bijection $\theta_X:X\to X^{(0)}$ sending $(y,x)$ to $(y,(x_i-c_i(x_{<i}))_i)$.
For a definable map $f:X\to A$ there is the natural corresponding function $f^{(0)}= f\circ \theta_X^{-1}$ from $X^{(0)}$ to $A$.
\end{defn}

\begin{defn}[Globally analytic cells]\label{cella}
Suppose that $X\subset Y\times L^n$ is an open cell over $Y$.
If $n=0$ then $X$ is a globally analytic cell over $Y$. For $n>0$, if the image of $X$ under the coordinate projection $p$ from $Y\times L^n$ to $Y\times L^{n-1}$ is a globally analytic cell, and if $c_{n,y}^{(0)}$ is globally analytic on $p(X^{(0)})_y$ for each $y\in Y$ in the sense of Definition \ref{glob}, then $X$ is called a globally analytic cell over $Y$.
\end{defn}


\begin{thm}[Globally analytic Cell Decomposition]\label{cda}
Given definable sets $Y$ and $X\subset Y\times L^n$ and a definable map $f:X\to L^s$, there exist $N>0$ and a definable bijection
$$
X\to X'\subset L_N^N\times X \subset L_N^N\times Y\times L^n
$$
over $X$ such that $X'$ is the disjoint union of a cell with empty interior and an open globally analytic cell $A$ over $L_N^N\times Y$ such that
$f^{(0)}_{a,y}$ is globally analytic on $A^{(0)}_{a,y}$ for each $a\in L_N^N$ and $y\in Y$.
\end{thm}

We define an expansion $\cL^*$ of $\cL$ similar to the one of (4.1) of \cite{CLR}, and to the one of Definition 6.1.7 of \cite{CLip}, by joining division and witnesses for henselian zeros and roots.
\begin{defn} \label{hf1}
Let $\cL^*$ be the expansion of $\cL\cup\{^{-1}\}$ obtained by joining to $\cL\cup\{^{-1}\}$ function symbols $(\cdot,\cdot)_e^{1/m}$ and $h_{m,e}$ for $e\geq 0$ and $m>1$, where on a henselian valued field $K$ of characteristic zero and with value group $\Gamma_K$ these functions are:
$$
(\cdot,\cdot)_e^{1/m}:K\times K_{e^2}\to K
$$
sends $(x,\xi)$ to the (unique) $m$-th root $y$ of
  $x$ with $\ac_{e}(y)\equiv \xi\bmod e\cM_K$ and $\ord(y)=z$, whenever simultaneously
  $\xi^m=\ac_{e^2}(x)$, $\ac_e(m)\not=0$,
  and $\ord(x)$ is divisible by $m$ in $\Gamma_K$, and to $0$
  otherwise;
$$
h_{m,e}:K^{m+1}\times K_{e^2} \to K
$$
sends $(a_0,\ldots,a_{m},\xi)$ to the unique $y$ satisfying
$\ord(y)=0$, $\ac_{e}(y)\equiv \xi\bmod e\cM_K$, and
$\sum_{i=0}^{m} a_{i} y^i=0$, whenever $\xi$ is a unit,
$\ord(a_i)\geq 0$, $\sum_{i=0}^m a_{i} \xi^i\equiv 0\bmod e^2\cM_K$, and
 $$
 f'(\xi)\not\equiv0\bmod e\cM_K
 $$
with $f'$
the derivative of $f$, and to $0$ otherwise.
\end{defn}

\begin{prop}\label{terms}
Given a definable function $f:X\to Y$, there exists a definable bijection $\lambda:X\to X'\subset X\times L_N^N$ over $X$ and a tuple of $\cL^*$-terms $h$ such that
$$
h(x') = f(x)
$$
for all $x\in X$ and with $x'= \lambda(x)$.
\end{prop}
\begin{proof}
By \cite{CLR}, Theorem 7.5, the proposition holds in a slightly different setting. Namely in \cite{CLR}, an extra value group variable is allowed as input in the function $(\cdot,\cdot)_e^{1/m}$ and in parametrizations  $\lambda$ runs over a Cartesian product of $L_N^N$ with the value group. Since in the present case the value group is simply $\ZZ$, the proposition as stated follows directly from \cite{CLR}, Theorem 7.5.
\end{proof}

\begin{proof}[Proof of Theorem \ref{cda}]
We proceed by induction on $n$. For $n=0$, there is nothing to prove. We will use the classical form of cell decomposition without global analyticity, which follows immediately from tameness and compactness, or alternatively, by Theorem 7.4 of \cite{CLR}. Suppose now that $n\geq 1$.
Let $\cL^*$ be the expansion of $\cL$ given by Definition \ref{hf1}. 
By Theorem \ref{terms}, we may suppose that $f$ is given by a tuple of $\cL^*$-terms $t_j$. We may focus on $t_1$ among the $t_j$.
We proceed now by induction on the complexity of the term $t_1$. Suppose that $t_1$ equals $h(v_{1},\ldots,v_{m})$ for some $L$-valued terms $v_{i}$ and a function symbol $h$ of $\cL^*$. By the classical form of cell decomposition and both ongoing inductions, we may assume that $X$ is already a globally analytic cell over $Y$, that the $v_{i,y}^{(0)}$ are globally analytic for each $y\in Y$, and that, for a chosen $M>0$, $|v^{(0)}_{i,y}|$ and $\ac_M(v^{(0)}_{i,y})$ are constant on each box contained in $X^{(0)}_y$.
Now, by choosing $M$ appropriately depending on $h$ (as explained in the proof of Lemma 6.3.15 of \cite{CLip} for each possibility for $h$),  the theorem follows.
\end{proof}

The following elementary lemma about compositions will often be used without mentioning.

\begin{lem}\label{prod} Let $n,m,r$ be integers.
Let $f: U \to V$  and $g: V  \to \cO_L$ be locally analytic functions on some open subsets $U\subset \cO_L^n$ and $V \subset \cO_L^m$.
Assume that
$f$ and $g$ satisfy $T_r$. 
Then the composition $g\circ f $ satisfies
$T_r$. 
\end{lem}
\begin{proof}
Just use that the Taylor polynomial of a composition corresponds to the composition of the Taylor polynomials (of respective degrees and up to a certain degree),
 and use a classical ultrametric calculation.
\end{proof}

The composition with power maps has been used in the context of real parametrizations by Yomdin, Gromov, Pila and Wilkie. We will use power maps similarly and introduce the following notation for convenience.

\begin{defn}\label{fNb}
For $f: A \subset \cO_L^m \to L^n$ a definable function, for any integer $N>0$ and any $b\in\cO_L^m$, write $A_{\star N,b}$ for the set of all $x\in\cO_L^m$ such that $bx^N:=(b_ix_i^N)_{i\in \{1,\cdots, n\}}$ lies in $A$, and write $f_{\star N,b}:A_{\star N,b}\to L^n$ for the function $x\mapsto  f(b x^N)$.
\end{defn}

Let us recall how convergent power series over $\cO_L$ may also be interpreted in other valued fields, even when they are non-complete or of higher rank. For $m\geq 0$, put $A_m=
\cO_L
\{ x_1,\ldots,x_m \}$, namely the ring of formal power series in $x$ over $\cO_L$ and converging on $\cO_L^m$. Write $\cF(X, Y)$ for the ring of $Y$-valued functions on $X$ for any sets $X,Y$.
Let $L'$ be a valued field with valuation ring $\cO_{L'}$ and maximal ideal $\cM_{L'}$. An analytic $\{A_m\}_m$--structure on $L'$ is the data of 
ring homomorphisms
$$
\sigma_m : A_m \to \cF(\cO_{L'}^m, \cO_{L'}),
$$
for all $m\geq 0$,
satisfying
 \begin{itemize}
 \item[(1)]\label{1)} $\sigma_0(\cM_L) \subset \cM_{L'}$,
 \item[(2)]\label{2)} $\sigma_{m}(x_i)=$ the $i$-th coordinate function
on $\cO_{L'}^m$ for  $i=1,\dots,m$,  and
 \item[(3)]\label{3)} $\sigma_{m+1}$ extends $\sigma_{m}$  with the natural inclusions $A_m\hookrightarrow A_{m+1}$ and $\cO_{L'}^m \hookrightarrow  \cO_{L'}^{m+1}:\xi\mapsto (\xi,0)$ inducing $\cF(\cO_{L'}^m, \cO_{L'}) \hookrightarrow \cF(\cO_{L'}^{m+1}, \cO_{L'})$.
\end{itemize}

We also consider one-sorted variants $\cL_1$ and $\cL_1^h$ of $\cL$.
\begin{defn} \label{hfL1}
For $K$ a henselian field, let $h_n : K^{n+1} \to K$ be the function that associates to $(a_0, \cdots , a_{n}, b) \in \cO_K$ the unique zero, $c$, of the polynomial $p(x) := a_n x^n + a_{n-1}x^{n-1} + \cdots + a_0$ that satisfies $|c-b| < 1$, if $|p(b)| < 1$ and $|p'(b)| =1$, and let $h_n$ output $0$ in all other cases. Corresponding to the choice of $\cL$ as either $\LPas^L$ or $\Lan^L$, let $\cL_1$ be the valued field language $(\cdot,^{-1},+,-,0,1,\mid )$ with coefficients from $L$, resp., the valued field language  together with function symbols for each element of $A_m$ for all $m\geq 0$. Let $\cL_1^h$ be $\cL_1$ together with the function symbols $h_n$ for all $n\geq 0$.
\end{defn}

\begin{thm}[\cite{CLip2}, Theorem 3.4.2]\label{qeL1h}
Let $L^{\rm alg}$  be an algebraically closed valued field with analytic $\{A_m\}_m$--structure. Then $L^{\rm alg}$ admits quantifier elimination in (the one-sorted language) $\cL_1^h$.
\end{thm}

The following lemma gives uniform bounds on Gauss-norms and is based on Lemma 6.3.9 of \cite{CLip}. Recall that the Gauss-norm of a power series is the supremum of the norms of the coefficients. Write $\cM_{L^{\rm alg}}$ for the maximal ideal of the valuation ring $\cO_{L^{\rm alg}}$ of $L^{\rm alg}$.


\begin{lem}\label{gauss}
Let $f_y:\cM_L^m \to L$ be a definable family of functions for $y$ in a definable set $Y$. Suppose that for each $y\in Y$, $f_y$ is given by a power series with coefficients in $L$ which converges on the associated set $B_{\rm as}:=(\cM_{L^{\rm alg}})^m$. Suppose further for each $y\in Y$ and each $i=1,\ldots,m$ that the partial derivatives $\partial f_y/\partial x_i$ have norm at most $1$ on $B_{\rm as}$.
Then there is a nonzero integer $M$ such that the Gauss-norm of $M(f_y-f_y(0))$ is at most one for each $y\in Y$.
\end{lem}
\begin{proof}
The proof is somewhat easier in the equicharacteristic zero case, but we will give a uniform treatment.
Our lemma is implied by compactness by the following more general and abstract result.
Let $L'$ be an algebraically closed valued field with an analytic structure in the sense of Definition 4.1.6 of \cite{CLip} and let $f:\cM_{L'}^m \to {L'}$ be given by a power series in $A_{0,m}$ of the separated Weierstrass system of the analytic structure. Suppose for each $i=1,\ldots,m$ that $\partial f/\partial x_i$ has norm at most $1$ on $\cM_{L'}$.
Then there is a nonzero integer $M$ such that the Gauss-norm of $M(f-f(0))$ is at most one. 
When $m=1$ this follows from the first part of the proof of Lemma 6.3.9 of \cite{CLip} (showing that $c$ is not infinitesimal).
The statement for general $m$ follows from induction on $m$, Property (v) of Definition 4.1.2 of \cite{CLip}, and compactness.
\end{proof}

The following corollary expresses that, for power series in our setting, bounded $C^1$-norm almost implies $T_1$.
\begin{cor}\label{cor:gauss}
Let $f_y:B_y\subset \cM_L^m \to L$ be a definable family of functions for $y$ varying in a definable set $Y$. Suppose that for each $y\in Y$, $B_y$ is a box and that $f_y$ is given by a power series with coefficients in $L$ which converges on the associated set $B_{y, {\rm as}}$. Suppose further that, for each $y\in Y$, the function $f_y$ has $C^1$-norm at most $1$ on $B_{\rm as}$.
Then there is a nonzero integer $M$ such that, for each $y\in Y$, the function $M f_y$ satisfies $T_1$ on $B_{y, {\rm as}}$.
\end{cor}
\begin{proof}
Let us first consider the situation for fixed $y$, so we may  write $f$ instead of $f_y$, and so on. Since $B$ is a Cartesian product, to prove this case we may furthermore assume $m=1$.

Up to a translation, we may suppose there exists a bijection $i: \cM_L \to B:x\mapsto ax$ for some nonzero $a\in \cO_L$, and that $f(0)=0$.
The function
$$
g: \begin{cases} \cM_L \longrightarrow  \cO_L  \\ x \longmapsto \frac{f(ax)}{a}\end{cases}
$$
is globally analytic and has $C^1$-norm at most $1$ on $\cM_{L^{\rm alg}}$, by the chain rule for differentiation.
By Lemma \ref{gauss}, the Gauss-norm of $Mg$ is at most $1$ for some $M>0$.  Let us write $\sum a_i x ^i$ for the series $Mf(x)$, and thus $\sum a_i a^{i-1}z^i$ for the series $Mg(z)$. By the bound on the Gauss-norm of $Mg$, one has
$$
|a_ia^{i-1}|\leq 1.
$$
For $x,y\in B_{\rm as}$ 
let us write $x=av$ and $y=aw$ for $v,w$ in $\cM_{L^{\rm alg}}$, and
$$
|Mf(x) - Mf(y)| = | a_1(x-y) + a_2a^2(v^2-w^2) + a_3a^3(v^3-w^3) + \ldots   |.
$$
Rewriting $v^i-w^i$ by $(v-w)(v^{i-1}+\ldots + w^{i-1})$, one thus finds
 $$
|Mf(x) - Mf(y)| \leq \max_i |a_i a^{i-1} a(v-w)| \leq |a(v-w)| = |x-y|.
$$
This proves the statement for fixed $y$. The general case follows from the case $m=1$ and the uniformity in $M$ given by Lemma \ref{gauss}.
\end{proof}

\begin{proof}[Proof of Theorems \ref{GYT}, \ref{GYTs} and Proposition \ref{gyts} for $r=1$]
We proceed by induction on $m$, the case of $m=0$ being trivial by (mixed) tameness, (see Section \ref{ex-tame}).
The proof will combine Corollary \ref{cor:invers}, Theorems \ref{Lip}, \ref{Lipmixed} and \ref{cda}, and will require going to an algebraic closure $L^{\rm alg}$ of $L$ to control the $C^1$-norms on $L^{\rm alg}$.

By using the two constants $0$ and $1$ in $k_L$ or in $L_N$ to realize disjoint unions, it is clear that we may proceed by working piecewise on $X$. Also, by induction on the dimension of $X$ and the dimension theory as in \cite{CLb}, we may replace $X$ by a definable subset whose complement in $X$ has dimension less than $m$.

We may therefore assume that we have a definable bijection
$$
h_0:Q_0  \to X,
$$
where $Q_0 \subset Y\times L_N^N\times \cO_L^m $ and the $Q_{0,y,a}$ are open in $\cO_L^m$ for each $y\in Y$ and $a\in L_N^N$, and such that
the $h_{0,y,a}$ are $C^1$ for each $y$ and $a$. Indeed, such a bijection can be found by (a basic form of) Theorem \ref{cda}.
By Theorem  \ref{terms}, we may suppose that the components of $h_0$ are given by $\cL^*$-terms $t_1,\ldots,t_n$, with notation from Definition \ref{hf1}.

Let $L^{\rm alg}$ be an algebraic closure of $L$ with norm extending the one on $L$. We now explain how to switch between $L$ and $L^{\rm alg}$, to improve $h_0$. This passage to the algebraic closure will preserve the necessary information by the term structure given by Proposition \ref{terms} for $L$ and by Theorem 3.4.3.(i) of \cite{CLip2} for $L^{\rm alg}$, and by the quantifier elimination result stated as Theorem \ref{qeL1h}.

The field $L^{\rm alg}$ has a natural $\cL^*$-structure and $\cL_1^h$-structure by Lemma 3.3.6 and Theorem 3.4.1 of \cite{CLip2} and Theorem 4.5.11 of \cite{CLip}; this $\cL^*$-structure expands the natural $\cL_1^h$-structure on $L^{\rm alg}$. Moreover, every $\cL^*$-term corresponds in $(L^{\rm alg},\cL_1^h)$ naturally to an $\cL_1^h$-term, where the variables from $L^{\rm alg}_N$ are replaced by (new) variables over $L^{\rm alg}$. For example, the $\cL^*$-term $(\cdot,\cdot)_e^{1/m}$ is interpreted by the function
$
(L^{\rm alg})^2 \to L^{\rm alg}
$
which sends $(x,w,a)$ to the (unique) $m$-th root $y$ of
  $x$ with $|y-w| < |ey|$, whenever simultaneously
  $|w^m-x|<|e^2x|$, $|m|\geq |e|$, and to $0$ otherwise, which is given by an $\cL_1^h$-term.

Let us now associate to the $\cL^*$-terms $t_j$ in $L^{\rm alg}$ the corresponding $\cL_1^h$-terms $v_j$ and consider them as $L^{\rm alg}$-valued functions on $(L^{\rm alg})^{S+m}$ for some $S\geq 0$.
We will mimic the proof of Corollary \ref{cor:invers} to make the $C^1$-norm of the $v_j$ small and then go back to $L$.
By Theorem \ref{qeL1h}, there is a finite quantifier free $\cL_1^h$-definable partition of $(L^{\rm alg})^{S+m}$ with pieces $A_s$ such that, up to neglecting lower dimensional parts and possibly permuting coordinates, we may suppose for each piece $A_s$ that $|\partial v_{1} / \partial x_1|$ is maximal among the $|\partial v_{j} / \partial x_i|$ for $j=1,\ldots,n$ and $i=1,\ldots,m$, and that it is either at most one or larger than $1$ on the whole of $A_s$.
In the latter case, we may assume by compactness, Propositions \ref{inj-cons} and \ref{inj-cons-mixed}, and reinterpreting back (as above) into $\cL_1^h$ if necessary, that the functions $v_{1,b,x_2,\ldots,x_m}$ are injective for each $(b,x)\in A_s$, with inverse $(v_{1,b,x_2,\ldots,x_m})^{-1}$ also given by an $\cL_1^h$-term by Theorem 3.4.3.(i) of \cite{CLip2}. Replacing the restriction of $v = (v_j)_j$ to $A_s$ with the function
$$
(b, w , x_2,\ldots,x_m) \mapsto v (b, (v_{1,b,x_2,\ldots,x_m})^{-1}(w) , x_2,\ldots,x_m),
$$
and by the chain rule for differentiation (with a similar calculation as in the proof of Corollay \ref{cor:invers}), it follows that we may suppose that the functions $v_{j,b}$ have $C^1$-norm bounded by one on each $A_{s,b}$.
Since moreover such terms are almost everywhere locally analytic, the $v_{j,b}$ may be assumed to be locally $T_1$.

Interpreting this data back in $L$, we obtain an improved $\cL$-definable function $h:Q\subset  Y\times L_N^N \times \cO_L^m\to \cO_L^n$.
Using cell decomposition for $h$ as  provided by Theorem \ref{cda}, it follows from Theorems \ref{Lip}, resp.~\ref{Lipmixed},  Corollary \ref{cor:gauss}, and by induction on $m$ applied to the graphs of the centers to replace $Q$ by $Q^{(0)}$ as in Definition \ref{cell0}, that one may assume that there is an integer $N'>0$ (with $N'=1$ in the equicharacteristic zero case) such that
$$
N'h:
\begin{cases}Q\subset  Y\times L_N^N \times \cO_L^m\longrightarrow N'X \\
(y,a,x) \longmapsto N'h(y,a,x)\end{cases}
$$
is as desired, but for $N'X$ instead of  $X$, with $N'$ coming from  the use of Theorem \ref{Lipmixed} and Corollary \ref{cor:gauss}.





We are done if $k_L$ is of characteristic zero, since then $N'=1$. Now suppose that $k_L$ is of positive characteristic.
In that case we shall make use of the fact that,   for $x$ close enough to $y$ in $\cO_{L^{\rm alg}}$, one has, with $p=p_L$,
\begin{equation}\label{powerp}
|x^p - y^p| = |(x-y) (x^{p-1}+\ldots+y^{p-1})|\leq |(x-y)p|,
\end{equation}
to gain a factor $|p|$.
Note also that it is always possible to increase the depth of the cell $Q$ over $Y\times L_N^N$, at the cost of  increasing $N$.

We first incease the depth of the cell $Q$ over $Y\times L_N^N$ by a factor which is a power of $p_L$, then replace the $h_{y,a}$ by composing with $M$-th powers for some $M$ which is a power of $p_L$, and restore the condition of having open cells around zero (which is possible since the group of $M$-th powers have finite index in $L^\times$). Using (\ref{powerp}) we observe that all conditions are met for the new $h$, for a large enough choice of the powers of $p_L$, depending only on $N'$.
This finishes the proof.
\end{proof}


\begin{lem}\label{gauss0}
Let $g:B\to L$ be a globally analytic function, where $B$ is the box $a\cM_L$ for some nonzero $a\in\cO_L$.
If for some $\lambda\in L^{\rm alg}$
$$
|g|\leq |\lambda| \mbox{ on $B_{\rm as}$}
$$
then, for all $i>0$,
$$
\Big|\frac{g^{(i)}}{i!}\Big| \leq \frac{|\lambda|} {|a|^i}  \mbox{ on $B_{\rm as}$.}
$$
\end{lem}
\begin{proof}
First assume $a=1=\lambda$. 
The assumptions imply that the $\sup$-norm of $g$ on $B_{\rm as}$ is at most $1$. Since the $\sup$-norm of $g$ on $B_{\rm as}$ coincides with the Gauss-norm of $g$, it follows that also the latter is at most one. Hence, the Gauss-norm and the $\sup$-norm of $g^{(i)}/i!$ on $B_{\rm as}$ are also at most $1$. In other words, $|g^{(i)}/i!|\leq 1$ on $B_{\rm as}$. The general case follows by applying the case $a=1=\lambda$ to the function $h: x\mapsto g(ax)/\lambda$ on $\cM_L$.
\end{proof}

By Legendre's formula  $|p_L|^i \leq |i!|$ and thus, for any positive integer $n$ divisible by $p_L$, 
\begin{equation}\label{legendre}
\frac{|n|^{r}}{|i!|}\leq 1
\end{equation}
for all integers $i>0$.

We now prove a key lemma allowing to go from $C^1$ to $T_r$.

\begin{lem}\label{gauss1a} Let a positive integer $r$ be given.  
In the equicharacteristic zero case, let $N=r$ and let $n=1$. In the mixed characteristic case, let $n$ be a positive integer sufficiently divisible by $p_L$ and let $N$ be a positive integer divisible by $n^r$.
Let $B$ be the ball $b\cdot(1+n\cM_L)$ for some nonzero $b\in\cO_L$. Let $g:B\to \cO_L$ be a globally analytic function whose $C^1$-norm is at most $1$ on $B_{\rm as}$. Let $h_N$ be the map $L^{\rm alg}\to L^{\rm alg}$ sending $x$ to $x^N$. Then, 
for any ball $B'\subset \cO_{L}$ with $h_N(B')\subset B$, the function $g\circ h_N$ satisfies $T_r$ on $B'_{\rm as}$.
\end{lem}
\begin{proof}
Let us write $g_N$ for $g\circ h_N$ with domain $D$ consisting of $x\in \cO_{L^{\rm alg}}$ such that $h_N(x)$ lies in $B_{\rm as}$.
By the chain rule and the product rule for differentiation, for any $x\in D$ and for $i$ with $0<i$, the derivative $g_N^{(i)}(x)$ is a finite sum of terms of the form $N^\beta x^{\alpha} g^{(\beta)}(x^N)$ for $0<\alpha$ and $0<\beta\leq i\leq \beta N$ and with
\begin{equation}\label{alphab}
\alpha = \beta N - i.
\end{equation}
Moreover, by Lemma \ref{gauss0} applied to $g'$, one has
$$
|g^{(\beta)}| \leq \frac{1}{|nb|^{\beta-1}}  \mbox{ on $B_{\rm as}$}.
$$
Hence, for $b'\in L^{\rm alg}$ with $|b'|^N =|b|$ and for $x$ in $D$ one has $|b'|=|x|$ and thus we find
\begin{eqnarray*}
\Big|N^\beta x^\alpha g^{(\beta)}(x^N) \Big| & \leq &  \Big|N^\beta x^{\beta N - i } g^{(\beta)}(x^N)\Big| \\
&=& \Big|N^\beta b'{}^{ \beta N - i}g^{(\beta)}(x^N)\Big|  \\
&\leq& \frac{|N|^\beta}{|n|^{\beta-1}} \cdot\frac{| b'|^{ \beta N - i} }{|b|^{\beta-1}}\\
&\leq& |N| |b'|^{N  - i}\\
&\leq& |n|^r |b'|^{N  - i}.
\end{eqnarray*}
Hence, for each $i>0$
\begin{equation}\label{gNi}
\Big|
\frac{g_N^{(i)}(x)}{i!}\Big|\leq \frac{|n|^r}{|i!|} |b'|^{ N  - i },
\end{equation}
on $B'_{\rm as}$.
Thus, the $C^r$-norm of $g_N$ is at most $1$ on $B_{\rm as}$ by (\ref{legendre}).
Now choose $B'$ and choose $x,b'\in B'_{\rm as}$. Develop $g$ around $b'$ into a series  $g(z+b') = \sum_i a_iz^i$. Then $|a_i|\leq  |b'|^{N - i}|n^r|/|i!|$ by (\ref{gNi}), which implies that $g$ is $T_r$ on $B'_{\rm as}$ as follows. First note that $|nb'| \geq |  x - b' |$. Using this and the bounds on the $|a_i|$, we have
\begin{eqnarray*}
\Big|g_N(x) - T^{<r}_{b', g_N}(x)\Big| &=& \Big|  \sum_{i\geq r} a_i ( x - b'   )^i     \Big|  \\
&\leq& \max_{i\geq r} \Big( |b'|^{N - i}(|n^r|/|i!|) \cdot |  x - b'    |^i \Big)\\
& \leq & \max_{i\geq r} \Big( |b'|^{N - i}(|n^r|/|i!|)\cdot  |  x - b'    |^{i-r}\cdot |  x - b'   |^{r} \Big)\\
&\leq&   \frac{|n|^{i}}{|i!|}  | x-b' |^r\\
&\leq& | x-b' |^r,
\end{eqnarray*}
where the last inequality follows from (\ref{legendre}).
This proves the lemma.
\end{proof}


The following is a multi-variable variant of Lemma \ref{gauss1a}.

\begin{prop}\label{gauss1} Let positive integers $r$ and $m$ be given. In the equicharacteristic zero case
set $n=1$ and $N=r$.  In the mixed characteristic case, let $n$ be sufficiently divisible by $p_L$ and $N$ be sufficiently divisible by $n$.
Let $B$ be the box $\prod_{i=1}^m b_i\cdot(1+n\cM_L)$ for some nonzero $b_i\in\cO_L$. Let $g:B\to \cO_L$ be a globally analytic function.
Let $h_N$ be the map $(L^{\rm alg})^m\to (L^{\rm alg})^m$ sending $(x_i)_i$ to $(x_1^N,\ldots, x_m^N)$.
Suppose that the $C^1$-norm of $g$ is at most $1$  on $B_{\rm as}$.
Then, for any box $B'\subset \cO_{L}^m$ such that $h_N(B')\subset B$, the function $g\circ h_N$ satisfies $T_r$ on $B'_{\rm as}$.
\end{prop}
\begin{proof}
As for Lemma \ref{gauss1a}.
\end{proof}

\begin{proof}[Proof of Theorems \ref{GYT}, \ref{GYTs} and Proposition \ref{gyts} for $r>1$]
We may take $g:P\to X$ with all properties of Theorems \ref{GYT}, \ref{GYTs} and Proposition \ref{gyts} with $r=1$, since this case has already been proved.
We work with fixed $y\in Y$ and $a\in L_N^N$ and omit $y$ and $a$ from the subscripts, explaining uniformity properties in $y$ and $a$ along the way.


Take $n$ and $N$ corresponding to Proposition \ref{gauss1} and our $r$ and $m$.
We may increase the depth of the cell $P$ to the depth $(n_i)_{i=1}^m$ with $n_i=n$.
By Proposition \ref{gauss1}, applied to the restrictions of $g$ to any box $B$ in its domain we find for any $c\in\cO_L^m$ that the function $ g_{\star N,c}$ is globally analytic and satisfies $T_r$ on  $B'_{\rm as}$ for any box $B'$ included in $B_{\star N,c}$.

\par
Similarly as at the end of the proof of the case $r=1$, after rewriting the function $g_{\star N,c}$, we get a definable function $\overline g:\overline P\subset Y\times L_{\overline N}^{\overline N}$ as in the case $r=1$ of Theorems \ref{GYT}, \ref{GYTs} and Proposition \ref{gyts}. Moreover, for each $y\in Y$ and $a\in L_{\overline N}^{\overline N}$, the map $\overline g_{y,a}$ satisfies $T_r$ on $B'_{\rm as}$ for each box $B'$ in its domain.
We claim that $\overline g$ is as desired in Proposition \ref{gyts} and Theorems \ref{GYTs} and \ref{GYT}.
There is only left to check that $\overline g$ satisfies $T_r$ globally. We still omit $y$ and $a$ from the notation. To show $T_r$ for $\overline g$ we will use $T_1$ for the above $g$. Choose $v$ and $w$ in the domain $\overline P$ of $\overline g$. Let $I$ be the set of those indices $i$ such that $v_i$ and $w_i$ have the same order, and $I^c$ its complement  in $\{1,\ldots,m\}$. Let $z$ be the intermediary tuple $(z_i)_i$ such that $z_i=v_i$ when $i\in I$ and $z_i=w_i$ when $i\in I^c$.  Then $w$ and $z$ lie in the same box contained in $\overline P$, and by $T_1$ for the above $g$, we have
\begin{eqnarray*}
|\overline g(v) -  T^{< r}_{w, \overline g} (v)  | & \leq  & \max (|\overline g(v) - \overline g(z)| , |\overline g (z) -  T^{< r}_{w, \overline g} (z) |, | T^{< r}_{w, \overline g} (z) - T^{< r}_{w, \overline g} (v) | ) \\
 & \leq  &
  \max (|\overline g(v) -  \overline g(z) | , |z-w|^r, | T^{< r}_{w, \overline g} (z) - T^{< r}_{w, \overline g} (v) | )  \\
 & \leq  &
  \max (|v^r-z^r| , |z-w|^r ) \\
& \leq  &
  |v-w|^r.
\end{eqnarray*}
The first of these inequalities follows from the ultrametric triangle property, the second from the fact that $\overline g$ satisfies $T_r$ on each box in its domain, the third inequality follows from the construction of $\overline g$ via $g_{\star N,c}$, property $T_1$ for $g$ and the fact that $N\geq r$. The fourth and final inequality follows from the fact that $|v^r-z^r|=|v-z|^r$ by construction of the point $z$. Indeed, $|v-z|=\max_{i\in I^c} |v_i-z_i| = \max_{i\in I^c} (|v_i|,|z_i|)$ and similarly $|v^r-z^r| = \max_{i\in I^c} (|v_i|^r,|z_i|^r)$.
This finishes the proof of Proposition \ref{gyts}, Theorem \ref{GYTs} and thus also of Theorem \ref{GYT}.
\end{proof}

\begin{remark}\label{rem:inj}
In the case where $m=1$ in Theorems \ref{GYT}, \ref{GYTs} and by observing their proof, one can further ensure in Theorems \ref{GYT}, \ref{GYTs} that the coordinate projection
$$
Y\times L_N^N\times \cO_L \to Y\times \cO_L
$$
is finite to one on $P=\{(y,i,x)\mid x\in P_{y,i}\}$.
\end{remark}

\begin{remark}
\label{struct}
Let us comment on how the reparametrization results can be generalized to other fields than the fields $L$ of this section, and to other languages than $\cL=\Lan^L$ or $\cL=\LPas^L$. The language $\cL$ can be interpreted naturally in many more fields than just in $L$, as explained below Definition \ref{fNb}, with still a well-understood geometry of the definable sets by Section 3.4 of \cite{CLip2}. Similarly, instead of $\cL=\Lan^L$ or $\cL=\LPas^L$, for $\cL$ we can take any analytic language formed by adding function symbols for the elements of a separated Weierstrass system as in \cite{CLip} to $\LPas$, further enriched with some constant symbols, and interpret it as an analytic structure on a henselian valued field $L'$ of characteristic zero, as in \cite{CLip}. Suppose that $(L',\cL)$ is such a more general structure.
If, furthermore, for sufficiently many $N>0$, $\cO_{L'}$ is a finite union of sets of the form $\lambda P_N(\cO_{L'})$ for $\lambda\in L'$ where $P_N(\cO_{L'})$ is the set of $N$-th powers in $\cO_{L'}$, and if $\cL$ has constant symbols for these $\lambda$, then we strongly expect Theorems \ref{GYT} and \ref{GYTs} and Proposition \ref{gyts} to go through on $(L',\cL)$ with similar proofs, where the results of \cite{CLip} can be used instead of the quoted results from \cite{CLR}.
\end{remark}

\subsection{A determinant estimate}\label{de}

Once we have the parametrizations of definable sets provided by Theorem \ref{GYT}, we may derive a result analogous to Lemma 3.1 of \cite{qjm}
and Lemma 2.1 of \cite{Marmon} similarly as Pila does in \cite{qjm}.
We now introduce the following notation, that will be used again in subsequent sections:
$$\Lambda_m (k) = \{\alpha \in \mathbb{N}^m ; \vert \alpha \vert = k \}, \ \
\Delta_m (k) = \{\alpha \in \mathbb{N}^m ; \vert \alpha \vert \leq k \}, $$
and
$L _m (k) = \# \Lambda_m (k)$,
$D _m (k) = \# \Delta_m (k)$.
Thus,
$L _m (k) = \binom{k + m - 1}{m - 1}$
and
$D _m (k) = \binom{k + m}{m}$.

\begin{lem}\label{detest}Fix $\mu \in \mathbb{N}$. Let $U$ be an open subset of
$L^m$ contained in a box which is the Cartesian product of $m$ closed ball of equal radius $\varrho \leq 1$.
Let $x_1$, \dots, $x_{\mu}$ be points in $U$, and
$\psi_1$,  \dots, $\psi_{\mu}$ be $C^r$-functions $U \to L$.
Assume
\begin{enumerate}
\item[(1)]The integer $r$ satisfies
$$D_m (r - 1) \leq \mu < D_m (r),$$
\item[(2)]The functions $\psi_i$ satisfy $T_r$ on $U$.
\end{enumerate}
Set $$\Delta = \det (\psi_i (x_j)).$$
Then
$$
\vert \Delta \vert \leq \varrho^e
$$
with
$$
e = \sum_{k = 0}^{r - 1} k L_m (k) + r (\mu - D_m (r - 1)).
$$
\end{lem}

\begin{proof}
By hypothesis (2), one may write
$$ \psi_i (x_j) = T_{x_1, \psi_i}^{\leq r - 1} (x_j) + R_{i, j} $$ with $$ R_{i, j} \leq \varrho^{r}.$$
Expanding
$ T_{x_1, \psi_i}^{\leq r - 1} (x_j)$ into the sum of
$D_m (r - 1)$ monomial terms of type $\displaystyle \frac{1}{\alpha !}\frac{\partial^\alpha \psi_i}{\partial x^\alpha}(x_1)(x_j-x_1)^\alpha$
one gets an expansion
of $\psi_i (x_j) $ into the sum of
$D_m (r - 1) + 1$ terms, the last one being $R_{i,j}$.
The columns of the matrix $(\psi_i(x_j))$ being indexed by $j$,
 we can write each column in $\Delta$, except for the first one which is $(\psi_i(x_1))_{i\in \{1,\cdots, \mu\}}$, as a sum of
$D_m (r - 1) + 1$ columns, and then expanding the determinant one may write
$\Delta$ as the sum of
$(D_m (r - 1) + 1)^{\mu-1}$ determinants $\Delta_{\ell}$. For each
determinant $\Delta_\ell$, we factor out from its columns the factors $(x_j-x_1)^\alpha$. This lets us write $\Delta_\ell$ as
a product of factors $(x_j-x_1)^\alpha$ and of a determinant
$\delta_\ell$ with columns  $\displaystyle \Bigl (\frac{1}{\alpha !}\frac{\partial^\alpha \psi_i}{\partial x^\alpha}(x_1) \Bigr )_{i\in \{1,\cdots,\mu\}}$, called of order $\vert \alpha \vert $, and columns $(R_{i,j})_{i\in \{1,\cdots, \mu\}}$.
Note that  if $\delta_{\ell} \not=0$, then $\delta_\ell$ cannot have two identical columns, and thus cannot have more than $L_m (k)$ columns of order $k$, $k \leq r - 1$.
Now $\vert \Delta_\ell\vert $ is maximized
when the number of columns of type  $\displaystyle \Bigl (\frac{1}{\alpha !}\frac{\partial^\alpha \psi_i}{\partial x^\alpha}(x_1)\Bigr )_{i\in \{1,\cdots,\mu\}}$ in $\delta_\ell$ is maximal, that is  for
$L_m(k)$ columns of order $k$,   $k\in \{1,\cdots,r-1\}$. Note,
by hypothesis (1), that the number of these columns is then
$\displaystyle \sum^{r-1}_{k=0}L_m(k)=D_m(r-1)\le \mu$.
In this case the degree for the monomial factored out
from $\Delta_\ell$
is $\displaystyle \sum_{k=1}^{r-1}kL_m(k)$, and thus of norm
$\le \varrho^{ \sum_{k=1}^{r-1}kL_m(k)}$, and the number of the remaining columns in $\delta_\ell$, which are of type $(R_{i,j})_{i\in \{1,\cdots, \mu\}}$, is minimal and equals $\mu-D_m(r-1)$.
By hypothesis (2) and the ultrametric  property of the norm, it follows that $\vert \delta_\ell \vert \le \rho^{r(\mu-D_m(r-1))}$.
Thus, for such a $\Delta_{\ell}$,
$\vert \Delta_{\ell} \vert \leq  \varrho^e$. Finally, again by the ultrametric property of the norm, the statement follows.
\end{proof}

\section{A $p$-adic analogue of the Pila-Wilkie Theorem}\label{sec:3}
\subsection{}
We give in this section a $p$-adic version of Pila-Wilkie's Theorem 1.10 of \cite{PiWi} in the form stated by Pila in Theorem 3.5 of \cite{PiSelecta}, that is the so-called block version.
Though the arguments involved in the proofs of both versions are the same, the block version has shown to be more useful in applications
(a reason for the effectiveness of the block version is that the image of a block under a semialgebraic map has a controlled number of $0$-dimensional blocks).
For instance it allows one to bound the number of points of given algebraic degree over $\QQ$ and bounded height.
In the $p$-adic context, such a byproduct of the $p$-adic block version of  Theorem 3.5 of \cite{PiSelecta} is still possible and given at the end of the section in Theorem \ref{algheights}.

We shall work in this section with the language $\cL = \cL_{an}^{\QQ_p}$.  Thus the $\cL$-definable subsets of $\QQ_p^n$ are exactly the subanalytic sets.
We say an $\cL$-definable subset $X$ of $\mathbb{Q}_p^n$ is  of dimension $k$ at a point $x$ if for every small open ball $B$ containing $x$,
$B \cap X$ is of dimension $k$. We say $X$ is of pure dimension $k$
if it is of dimension $k$ at each of its points. 

Let $X$ be an $\cL$-definable subset of $\mathbb{Q}_p^n$.
One defines $X^{\mathrm{alg}}$  as the union of all semialgebraic subsets of $X$ of pure (strictly)
positive dimension. Note that this description of $X^{\mathrm{alg}}$ coincides with the one given in introduction,
since by the curve selection lemma, a point $x$ of a given
semialgebraic set $X$ of positive dimension at $x$ is always contained in an algebraic curve $C$ such that $C (\mathbb{Q}_p) \cap X$
is of positive dimension at $x$. We denote by $X(\ZZ,T)$
the set of points $(x_1,\cdots, x_n)\in X\cap \ZZ$
with $\vert x_i \vert_\RR\le T $, for $i\in \{1,\cdots,n\}$.

Let $Z$ be an $\cL$-definable subset of $\mathbb{Q}_p^n \times \mathbb{Q}_p^m$
and denote by $Y$ the projection of $Z$  on  $\mathbb{Q}_p^m$.
For $y$ in $Y$, we denote by $Z_y$ the fiber of $Z \to Y$ at $y$.
We view $Z$ as a  family of definable subsets $Z_y$ of
$\mathbb{Q}_p^n$
parametrized by $Y$.
We call $Z$ a definable family  of
definable subsets of $\mathbb{Q}_p^n $.

We begin this section by a simple but useful remark relating the norm $\vert \ \vert_\RR$ and
the $p$-adic norm $\vert \ \vert$.

\begin{remark}\label{bound} Let $x \in \ZZ\setminus\{0\}$ and $T>1$ be such that
 $\vert x \vert_\RR \le  T$. Choosing $r\in \NN$ such that
$T  \le p^r$, since $\vert x \vert_\RR  \le p^r$, one obtains $\vert x\vert \ge p^{-r}  \ge T^{-1}$.
\end{remark}


\begin{lem}\label{basicest}
Fix positive integers $m < n$. Then for every integer $d \geq 1$ there exists
an integer $r = r (m, n, d)$ and positive constants $\varepsilon  (m,n,d)$ and
$C  (m, n, d)$ such that the following holds.
For every real number $T >1$, for every open subset $U$ of $\mathbb{Z}_p^m$, for every
locally analytic mapping $\psi =  (\psi_1,  \dots, \psi_{n}) : U \to \mathbb{Q}_p^n$
such that the functions $\psi_\ell$ satisfy condition $T_r$, the subset $\psi (U)(\ZZ,T)$
is contained in the union of at most $$C  (m, n, d)  \, T^{ \varepsilon (m, n, d)}$$ hypersurfaces of
degree $\leq d$.
Furthermore, $\varepsilon (m, n, d) \to 0$ as $d \to \infty$.
\end{lem}

\begin{proof} With the notation of Section \ref{de}, set
$\mu = D_n(d)$ and fix $r$ such that
$D_m (r - 1) \leq \mu < D_m (r)$.
Note that $r$ is unique, so we can denote it by
$r(m, n, d)$.
Let $\psi =  (\psi_1,  \dots, \psi_{n}) : U \to \mathbb{Z}_p^n$, with $U$ an open subset of $\mathbb{Z}_p^m$,
such that the functions $\psi_i$ satisfy the condition  $T_r $.
Let $\Sigma\subset  \mathbb{Z}_p^m$ be a ball of radius $\varrho < 1$. Consider $\mu$ points (possibly with repetition) $P_1$, \dots, $P_{\mu}$
in $\Sigma \cap U \cap \psi^{-1} (\psi(U)(\ZZ,T))$
and the $\mu$ monomials $\psi^{\alpha}:=\psi_1^{\alpha_1}\cdots \psi_n^{\alpha_n}$, $\alpha \in \Delta_n (d)$.
By Lemma \ref{prod} the mappings $\psi^{\alpha}$, for
$\alpha \in \Delta_n (d)$, satisfy $T_r$ as well as the mappings $\psi_\ell$,
$\ell\in \{1,\cdots, n\}$.
By Lemma \ref{detest} it follows that
\begin{equation}\label{upbis}
\vert \det (\psi^\alpha (P_j))  \vert \leq   \varrho^{e}
\end{equation}
with
$$e = e (m,n, d) = \sum_{k = 0}^{r - 1} k L_m (k) + r (\mu - D_m (r - 1)).$$
On the other hand, since $\vert \psi_\ell(P_j)\vert_\RR \le T$, for all $j\in \{1, \cdots, \mu\}$ and all
$\ell\in \{1, \cdots, n\}$, after expanding  the determinant $ \det (\psi^\alpha (P_j))$ one gets a sum of $\mu!$ integers, each
of them having a real norm $\le T^V$, with $V = \sum_{k = 0}^d k L_n (k)$. It follows that the integer
$ \det (\psi^\alpha (P_j))$ has a real norm   $\le \mu! T^V$, and finally by Remark \ref{bound}, one has
under the condition $\det (\psi^\alpha (P_j)) \not= 0$,
\begin{equation}\label{downbis}
\vert \det (\psi^\alpha (P_j))  \vert   \ge   \mu!^{-1}T^{-  V}.
\end{equation}

By putting together (\ref{upbis}) and (\ref{downbis}),
it follows that if
\begin{equation}\label{boundonrhobis}\varrho  < \mu!^{-1/e}  T^{-V/e}
\end{equation}
then $\det (\psi^\alpha (P_j)) = 0$.

Now the end of the proof  is quite similar to  that of Lemma 1 of \cite{BP}.
For $\varrho>0$ as in (\ref{boundonrhobis}), note that the
matrix $$ A(P_1, \cdots, P_\mu) = \Bigl(\psi^\alpha (P_j)\Bigr) $$
has rank $\leq \mu - 1$, with $\alpha$ running over $\Delta_n (d)$
and for any  $P_1,\cdots, P_\mu
\in \Sigma \cap U \cap \psi^{-1} (\psi(U)(\ZZ,T))$.
Say that, for instance,  in this matrix the columns are indexed by $j$.
Let $a$ be the maximal
rank of $ A(P_1, \cdots, P_\mu)$ over all $P_1,\cdots, P_\mu
\in  \Sigma \cap U \cap \psi^{-1} (\psi(U)(\ZZ,T))$ and
let $M=(\psi^\alpha(P_j))_{\alpha\in I, j\in \{1, \cdots, a\}}$
be of rank $a$, for some fixed $P_1,\cdots, P_a\in  \Sigma \cap U \cap \psi^{-1} (\psi(U)(\ZZ,T))$
and some  $I\subset \Delta_n(d)$ of cardinality $a$.
Since $a<\mu$, we can choose $\beta\in \Delta_n(d)\setminus I$.

Let us denote by $f(x)$ the determinant of the matrix
$(M' x^\gamma)_{\gamma\in I\cup \{\beta\}}$, where
$x=(x_1, \cdots, x_n)$ and  $M'$ is $M$ augmented
by the line $\psi^\beta(P_j)_{j\in \{1, \cdots, a\}}$.
The polynomial $f$ is not zero since the coefficient of $x^\beta$ in $f$ is
the non zero minor $\det(M)$ and the degree of $f$ is at most $d$.
Furthermore for any $u \in  \Sigma \cap U \cap \psi^{-1} (\psi(U)(\ZZ,T))$ we have
$f(\psi(u))=0$, by definition of the maximal rank $a$.

Since there exists a constant $C'$ depending only on $m, n, d$ (and $p$), such that
$\mathbb{Z}_p^m$ is covered with $\leq C' \, T^{m V/e}$
balls of radius $\varrho$ such that (\ref{boundonrhobis}) holds, we get the required result,
since, by a straightforward computation done in \cite{qjm}, p. 212 (the constant $B$ of \cite{qjm}
being our constant $e$), for fixed $m < n$, $m V/e \to 0$ as $d \to \infty$.
\end{proof}

\begin{prop}\label{ml}
Let $Z$ an $\cL$-definable family  of
$\cL$-definable subsets of $\mathbb{Z}_p^n $  parametrized by $Y \subset \mathbb{Q}_p^{\ell}$.
Assume all fibers have dimension $ < n$.
Let $\varepsilon > 0$. There exists an integer $d = d (\varepsilon, m, n)$ and a positive real number
$C (Z, \varepsilon)$
such that, for every $y$ in $Y$ and every $T>1$, the set
$Z_y(\ZZ,T)$ is contained in the union of at most
$C (Z, \varepsilon) \, T^\varepsilon$ algebraic hypersurfaces of degree at most $d$.
\end{prop}

\begin{proof} The argument is quite similar to the one in Proposition 6.2 of \cite{PiWi}.
We assume that all fibers of $Z$ have dimension $m<n$. Take $d=d (\varepsilon, m, n)$  large enough in order to have
$\varepsilon(m,n,d) < \varepsilon$ and take $r=r(m,n,d)$, where $\varepsilon(m,n,d)$ and
$r(m,n,d)$ are given in Lemma \ref{basicest}. By Theorem \ref{GYT} there exist an integer
$K$ and a definable family $(g_{y,i})_{y\in Y, i\in \{1, \cdots,K\}}$ of locally analytic functions
$$ g_{y,i}: P_{y,i}\to Z_y$$
satisfying $T_r$ and such that
$$ \bigcup_{i=1}^K g_{y,i}(P_{y,i})=Z_y.$$
By Lemma \ref{basicest}, since $\dim(Z_y)<n$, for every $i\in \{1, \cdots, K\}$, the set $g_{y,i}(P_{y,i})(\ZZ,T)$
is contained in  $C(m,n,d)T^\varepsilon$ algebraic hypersurfaces of degree at most $d$, with $C(m,n,d)$ as in Lemma \ref{basicest}.
We set $C(Z,\varepsilon)=K C(m,n,d)$.
\end{proof}



We now come to the main result of this section, the $p$-adic version of Theorem 3.5 of \cite{PiSelecta}. We first
define our notion of block, essentially in the same way as in \cite{PiSelecta}, up to connectedness.

\begin{defn}
A block $A\subset \QQ_p^n$ is either a singleton, or,
a smooth definable set of pure dimension $d>0$ which is contained in a smooth semialgebraic set of pure dimension $d$.

In particular, for a block $W$ of positive dimension, one has $W^{\mathrm alg}=W$.
Note that the interior $int(X)$ of a definable set $X$ of dimension $n$ in $\QQ_p^n$ is a block, since $int(X)$ is the intersection of itself with the semialgebraic set $\cO_K^n$. On the other hand the regular part of a definable set is not always a block.
A family of blocks $W \subset \QQ_p^n \times Y$ is a definable set whose fibers $W_y$, for $y\in Y$,
are blocks in $\QQ_p^n$.
\end{defn}


In Proposition \ref{blocks}, we consider integer points of bounded height.
We deduce from it Theorem \ref{algheights}, which is about rational points,
and  we finally prove the rational version  of Proposition \ref{blocks} in Theorem \ref{blocksQ}.

\begin{prop}\label{blocks}
Let $Z\subset \QQ_p^{n+\ell}$ be an ${\cal L}$-definable family of ${\cal L}$-definable subsets of $\QQ_p^n$ parametrized
by a definable set $Y\subset \QQ_p^{\ell}$. Let $\varepsilon >0$ be given. There exist $s=s(\varepsilon,n)\in \NN$, a constant
$C(Z,\varepsilon)$ and a
family of blocks $(W_{y,\sigma})_{(y,\sigma)\in \QQ_p^\ell\times \QQ_p^s}\subset \QQ_p^{n\ell} \times \QQ_p^s$
such that for any $y\in Y$, for any $T>1$
$$Z_y(\ZZ,T) \subset \bigcup_{\sigma \in {\cal S}} W_{y,\sigma},$$
for ${\cal S}={\cal S}(Z,\varepsilon, T) \subset \QQ_p^s$ of cardinal less than $C(Z,\varepsilon)T^{\varepsilon}$.
In particular, for all
$y\in Y$, denoting by $W_y^\varepsilon$ the union over $\sigma$ in $\QQ_p^s$
of the sets $W_{y,\sigma}$ with $\dim(W_{y,\sigma})>0$, one has $W_y^\varepsilon\subset (Z_y)^\mathrm{alg}$ and,
$$ \# (Z_y\setminus W_y^\varepsilon)(\ZZ,T)\le C(Z,\varepsilon) T^{\varepsilon}\ \mbox{for all $T>1$}. $$
\end{prop}

\begin{proof} We follow the strategy of  Pila's proof of Theorem 3.5  in \cite{PiSelecta}. Since we want to  bound the density
of integers points of $Z$, one can assume that $Z\subset  \ZZ_p^n\times Y$.
Since our result is true for the union of two families when it is true for each family, one can assume that the dimension
$k$ of the fibers $Z_y$ of $Z$ is constant and that these fibers are of pure dimension. We proceed by induction on $k$.

For $k=0$, the family $Z$ has a finite number of points for fibers, and this number is bounded with respect to the parameter $y$, by a constant
depending only on the set $Z$. It follows that $Z$ itself is a finite union of families of blocks.

Assume now that our statement is true for definable families of fiber dimension $\le k-1$ and consider $Z$ a definable family
in $\ZZ_p^{n+\ell}$ of fiber dimension $k\ge 1$.

We first remark that one can easily assume that the fibers $Z_y$ are not of maximal dimension in $\ZZ_p^n$. For this let us denote
by $Z^0_y$ the regular part of a fiber $Z_y$, that is the set of points of $Z_y$ at which $Z_y$ is smooth. Now if $k=n$, the definable
subset $Z'=\{(x,y)\in Z; x\in Z^0_y\}$ of $\ZZ_p^{n+\ell}$ is a family of blocks, since a fiber is given by the intersection of the semialgebraic
set $\QQ_p^{n}$ with the definable set $Z^0_y$. Finally since the definable set $Z\setminus Z'$ has fiber dimension $\le n-1$, we can apply to
this set the induction hypothesis, and obtain our statement in this case.

From now on we assume that $k<n$, and we fix $\varepsilon>0$.
By Proposition \ref{ml}, for any $y\in Y$ and for any choice of $k+1$ coordinates in $\QQ_p^n$, the projection
$\pi(Z_y)$ onto the corresponding $\QQ_p^{k+1}$ subspace of $\QQ_p^n$ is such that
$\pi(Z_y)(\ZZ,T)$ is contained in the union of at
most $C(Z,\varepsilon)T^{\varepsilon/M}$ hypersurfaces of degree at most $d=d(\varepsilon,k,n)$, with $M=n!/(k+1)!(n-k-1)!$,
and with some constant $C(Z,\varepsilon)$ depending only on $Z$ and $\varepsilon$. Let us denote by $\Sigma\subset \QQ_p^{n+\ell+s}$, $s=s(\varepsilon,n)$,
the family of
algebraic sets of $ \QQ_p^n$ defined by intersecting the cylinders in $\QQ_p^n$ over algebraic hypersurfaces of degree at most $d$ in $\QQ_p^{k+1}\subset \QQ_p^n$, for
every choice of $k+1$ coordinates in $\QQ_p^n$.
The dimension of a given fiber of $\Sigma$ depends on the transversality of the $M$  cylinders over hypersurfaces in $\QQ_p^{k+1}$ that give this fiber by intersecting each other, but this dimension is at most $k$, by transversality of the supplementary coordinates in $\QQ_p^n$
for each choice of $k+1$ coordinates.
We now consider the family $Z \times \QQ_p^s$ in $\QQ_p^{n+\ell+s}$ which, for the sake of simplicity, we shall still denote by $Z$. The set $Z_y(\ZZ,T)$ is contained in the intersection of the
cylinders over the hypersurfaces of $\QQ_p^{k+1}$ containing the points of $\pi(Z)(\ZZ,T)$, that is,
$Z_y(\ZZ,T)$ is contained
in at most
$C'(Z,\varepsilon)T^{\varepsilon}$ fibers of $\Sigma_y$ for some constant $C'(Z,\varepsilon)$.

Now, as in the proof of Theorem 3.5 of \cite{PiSelecta}, we stratify $Z\cap \Sigma$ in the following way.
Let us consider the definable family
$$ Z_1=\{(x,y,\sigma)\in Z\cap \Sigma\subset \QQ_p^{n+\ell+s}; x\not\in \hbox{reg}_k((Z\cap \Sigma)_{y,\sigma})\},$$
 where $\hbox{reg}_k ((Z\cap \Sigma)_{y,\sigma})$ is the regular subset of $(Z\cap \Sigma)_{y,\sigma}$ of dimension $k$, that is to say the set of points
 of $(Z\cap \Sigma)_{y,\sigma}$ in the neighbourhood of which $(Z\cap \Sigma)_{y,\sigma}$ is smooth and of dimension $k$. Then $Z_1$ is a definable family with fiber dimension
 $<k$ and we can apply the induction hypothesis : there exists a family of blocks $W^1\subset \QQ_p^{n+ \ell+s_1}$ such that
 for all $y\in Y$, $Z^1_y(\ZZ,T) \subset \bigcup_{\sigma\in {\cal S}_1} W_{y,\sigma}^1$, with
 $s_1=s_1(Z_1,\varepsilon)$ and ${\cal S_1}\subset \QQ_p^{s_1}$
 of cardinal less than  $C(Z_1,\varepsilon)T^{\varepsilon}$.

 The same kind of data $s_2$, $\cal S_2\subset \QQ_p^{s_2}$, $W^2$,  $C(Z_2, \varepsilon)$ and
 $s_3$, $\cal S_3\subset \QQ_p^{s_3}$, $W^3$,  $C(Z_3, \varepsilon)$
  are in the same way provided by the induction hypothesis for the families
$$ Z_2=\{(x,y,\sigma)\in Z\cap \Sigma\subset \QQ_p^{n+ \ell+s}; x\not\in \hbox{reg}_k((\Sigma)_{y,\sigma})\},$$
$$ Z_3=\{(x,y, \sigma)\in Z\cap \Sigma\subset \QQ_p^{n+ \ell+s}; x\not\in \hbox{reg}_k((Z)_{y,\sigma})\},$$
since these two families have fiber dimension $<k$ as well as the family $Z_1$.

Now observe that for $(x,y,\sigma)$ a point of $Z\cap \Sigma$ not in the family $Z_1\cup Z_2\cup Z_3$, a sufficiently small semialgebraic neighbourhood
$\Sigma_{y,\sigma}\cap B(x,\eta)$ of $x$ coincides with  $Z_{y,\sigma}\cap B(x,\eta)$.  The family $(Z\cap \Sigma)\setminus (Z_1\cup Z_2\cup Z_3)$ is therefore
a family of blocks with fibers of dimension $k>0$, such that, for $y\in Y$, a union of not more than
$C'(Z,\varepsilon)T^{\varepsilon}$ of them contains the whole set $Z_y(\ZZ,T) \setminus  (Z_1\cup Z_2\cup Z_3)$.

Denoting $C''(Z,\varepsilon)=C(Z_1,\varepsilon)+C(Z_2,\varepsilon) + C(Z_3,\varepsilon)$ and considering that the parameter spaces $\QQ_p^s$, $\QQ_p^{s_i}$, $i=1,\cdots,3$, are all
contained in a single parameter space, also denoted $\QQ_p^s$ for simplicity,  one obtains that, for any $y\in Y$, the set
 $Z_y(\ZZ,T)$ is contained in at most $C'(Z,\varepsilon)T^{\varepsilon}$ fibers of $\Sigma_y$ over $\QQ_p^s$, each of them being decomposed in at most
 $1+C''(Z,\varepsilon)T^{\varepsilon}$ family of blocks, providing the existence of the desired $W$ and ${\cal S}$. To conclude for the final statement of the proposition, observe again that $W_y^\varepsilon\subset (Z_y)^\mathrm{alg}$ follows from the definition of $(Z_y)^\mathrm{alg}$ and that blocks of dimension zero are singletons by definition, which implies the final bound of the proposition.
 \end{proof}

Before applying Proposition \ref{blocks} to algebraic points in $\QQ_p^n$ of bounded algebraic
degree over $\QQ$ and bounded height, let us recall the notion of polynomial height defined in the introduction that will both encode classical height and algebraic degree over $\QQ$.

For $a=r/s\in \QQ$, with $r$ and $s$ relatively prime integers, let
$$h(a) :=\max \{ \vert r \vert_\RR, \vert s \vert_\RR   \}$$
and for $a=(a_0, \cdots, a_k)\in \QQ^{k+1}$, we set
$$H_0(a ):=\max \{ h(a _0), \cdots, h(a_k) \}.$$
Now for  $k \in \NN\setminus \{0\}$ and  for $x\in \QQ_p$, we denote by $H^{poly}_k(x)$ the following element of $\NN\cup\{+\infty\}$
$$H^{poly}_k(x):=\inf\{ H_0(a); a=(a_0,\cdots, a_k)\in \QQ^{k+1}\setminus \{0\} , \sum_{j=0}^k a_jx^j=0   \},$$
$$H^{poly}_{k,\ZZ}(x):=\inf\{ H_0(a); a=(a_0,\cdots, a_k)\in \ZZ^{k+1}\setminus \{0\} , \sum_{j=0}^k a_jx^j=0   \},$$
and for $x=(x_1,\cdots, x_n)\in \QQ_p^n$, we finally set
$$H^{poly}_k(x):= \max\{ H^{poly}_k(x_i), i=1, \cdots, n\},$$
$$H^{poly}_{k,\ZZ}(x):= \max\{ H^{poly}_{k,\ZZ}(x_i), i=1, \cdots, n\}.$$



For $Z$ a definable subset of $\QQ_p^n$, $k\in \NN\setminus \{0\}$ and $T>1$ a real number,
 we denote by  $Z(k,T)$ the set of points $x\in Z$ such that
$H^{poly}_k(x)\le T$, and by $Z_\ZZ(k,T)$ the set of points $x\in Z$ such that
$H^{poly}_{k,\ZZ}(x)\le T$.

\begin{thm}\label{algheights}
Let $n, \ell,k$ be nonnegative integers.
Let $Z\subset \QQ_p^{n+ \ell}$ be an ${\cal L}$-definable family of ${\cal L}$-definable subsets of $\QQ_p^n$ parametrized
by a definable set $Y\subset \QQ_p^{ \ell}$. Let $\varepsilon >0$. There exists $s=s(\varepsilon,n)$, a constant
$C(Z,\varepsilon,  k)$ and a family of blocks $V\subset \QQ_p^{n+ \ell}\times\QQ_p^s$ such that for any
$y\in Y$, for any $T>1$
$$Z_y(k,T) \subset \bigcup_{\sigma \in {\cal S}} V_{y,\sigma},$$
for ${\cal S}={\cal S}(Z,\varepsilon,k,T) \subset \QQ_p^s$ of cardinal less than
$C(Z,\varepsilon,k)T^{\varepsilon}$.
In particular, for any $y\in Y$, denoting by $V^\varepsilon_y$ the union over $\sigma\in \QQ_p^s$ of the $V_{y,\sigma}$ of dimension $>0$, one has $V_y^\varepsilon \subset (Z_y)^\mathrm{alg}$ and
$$\# (Z_y\setminus V^\varepsilon_y)(k,T)\le C(Z,\varepsilon,k) T^{\varepsilon}. $$
\end{thm}

\begin{proof}
To prove this statement for all $k$ it is obviously enough to prove the similar statement for all $k$ where rational points are replaced by integer
points, that is to say it is enough to work with $Z_{y,\ZZ}(k,T)$ instead of  $Z_y(k,T)$.
For this goal, let us now consider
$$A_{n,k} =\{ (\xi,x,y)\in (\QQ_p^{k+1}\setminus \{0\})^n\times \QQ_p^{n+ \ell} ;  \ \sum_{j=0}^k\xi_{i,j}x_i^j=0,\ i=1,\cdots,n\}$$
and
$$Z_{n,k} =\{ (\xi,x,y)\in A_{n,k}  ; (x,y)\in Z \}.$$

Consider the projection $\pi_1:A_{n,k}\to (\QQ_p^{k+1})^n\times \QQ_p^{ \ell}$
defined by $\pi_1(\xi,x,y)=(\xi,y)$, let us write $U$ for $\pi_1(A_{n,k})$, and let us denote
by $\pi_2: A_{n,k}\to \QQ_p^n\times Y$ the projection defined by $\pi_2(\xi,x,y)=(x,y)$.

By definable choice, and since $\pi_1$ has fibers of size at most $k^n$, there exist
$k^n$ semi-algebraic maps $\Psi_i:U\to A_{n,k}$ which are  sections of $\pi_1$ and
such that the union of the graphs of the $\Psi_i$ equals $A_{n,k}$.
Hence,
\begin{equation*}
Z\subset \bigcup_{i=1}^{k^n} \pi_2(\Psi_i( \pi_1(   Z_{n,k}) ) ) )
\end{equation*}
and thus, by construction, one has for any $y\in Y$ that
\begin{equation}\label{union}
Z_{y,\ZZ}(k,T)\subset \bigcup_{i=1}^{k^n} (\pi_2(\Psi_i ( [ \pi_1(   Z_{n,k})]_y(\ZZ,T) ,y  )  ) )_y.
\end{equation}

Now, given $\varepsilon >0$ and applying Proposition \ref{blocks} to the definable family $\pi_1(Z_{n,k})$ with parameter $y\in Y$, we obtain a family of blocks
$W\subset \pi_1(Z_{n,k})\times \QQ_p^s$, such that for any $T>1$ and any $y\in Y$,
$[\pi_1(Z_{n,k})]_y (\ZZ,T) \subset \bigcup_{\sigma \in {\cal S}_0} W_{y,\sigma},$
for ${\cal S}_0 = {\cal S}_0(Z,\varepsilon,k, T) \subset \QQ_p^s$ of cardinal less than
$CT^{\varepsilon}$ for some $C$.

Since the maps $\pi_2$ and $\Psi_i$ are semi-algebraic, by the definition of blocks, and by dimension theory for $\cL$-definable sets, there exist integers $M=M(Z,\varepsilon,k)$ and $s'$ and a family of blocks $V\subset Z\times \QQ_p^{s'}$ such that any set of the form  $(\pi_2(\Psi_i(W_\sigma)))_{y}$ for any $\sigma\in \QQ_p^s$ and any $y\in Y$, can be written as the union of no more than $M$ blocks of the form $V_{y,\sigma'}$ for $\sigma'\in \QQ_p^{s'}$.
Combining with (\ref{union}) and with the information we have about ${\cal S}_0$, the existence of ${\cal S}$ with the desired properties follows for this $V$ and for any $T>1$, with $C(Z,\varepsilon, k)= M k^n C$. One concludes as for the proof of Proposition \ref{blocks}.
\end{proof}



Finally note that Theorem \ref{algheights} implies in particular  the following rational version of Proposition \ref{blocks}, that
differs only in its last line from Proposition \ref{blocks}.

\begin{thm}\label{blocksQ}
Let $Z\subset \QQ_p^{n+ \ell}$ an ${\cal L}$-definable family of ${\cal L}$-definable subsets of $\QQ_p^n$ parametrized
by a definable set $Y\subset \QQ_p^{ \ell}$. Let $\varepsilon >0$. There exist $s=s(\varepsilon,n)\in \NN$, a constant
$C(Z,\varepsilon)$ and a
family of blocks $W\subset \QQ_p^{n+ \ell}\times \QQ_p^s$
such that for any $y\in Y$, for any $T>1$
$$Z_y(\QQ,T) \subset \bigcup_{\sigma \in {\cal S}} W_{y,\sigma},$$
for ${\cal S}={\cal S}(Z,\varepsilon, T) \subset \QQ_p^s$ of cardinal less than $C(Z,\varepsilon)T^{\varepsilon}$. In particular, for all
$y\in Y$, denoting  $W_y^\varepsilon$ the union over $\sigma\in \QQ_p^s$
of the $W_{y,\sigma}$ of dimension $>0$, one has $W_y^\varepsilon\subset (Z_y)^\mathrm{alg}$ and
$$ \# (Z_y\setminus W_y^\varepsilon)(\QQ,T)\le C(Z,\varepsilon) T^{\varepsilon}. $$
\end{thm}

\section{A geometric analogue of results of Bombieri-Pila \cite{BP} and Pila \cite{pila_ast}}\label{sec:4}
\subsection{}\label{nr} In this section we shall work over the field $K = \CC\llp t \rrp$. Note however that all our results remain valid with identical proofs when $\CC$
is replaced by  any algebraically closed field of characteristic zero.

For each positive integer $r$ we denote by
$\mathbb{C} [t]_{< r}$ the  set of complex polynomials of degree
$<r$.
Let $A$ be a subset of
$\mathbb{C} \llp t \rrp^n$.
We denote by $A_r$ the set
$A \cap (\mathbb{C} [t]_{< r})^n$ and
by $n_r(A)$ the dimension of the Zariski closure of
$A_r$ in
$(\mathbb{C} [t]_{< r})^n \simeq \mathbb{C}^{nr}$.
Similarly, when $X$ is an algebraic subvariety of
$\mathbb{A}^n_{\mathbb{C} \llp t \rrp}$,
we shall write $X_r$ for  $(X(\mathbb{C} \llp t \rrp)_r$ and
$n_r(X)$ for  $n_r(X(\mathbb{C} \llp t \rrp)$.

We have the following basic estimate, which is the best possible when $X$ is linear:
\begin{lem}\label{trivest}Let $X$ be an algebraic subvariety of
$\mathbb{A}^n_{\mathbb{C} \llp t \rrp}$ of dimension $m$.
Then, for any $r >0$,
\[
n_r (X) \leq r m.
\]
\end{lem}

\begin{proof}Up to a $\mathbb{C}$-linear coordinate change, there is a coordinate projection
$$
p:\mathbb{A}^n_{\mathbb{C} \llp t \rrp}\to \mathbb{A}^m_{\mathbb{C} \llp t \rrp}
$$
whose restriction to $X$ has finite fibers. The projection $p$ induces a map
$p_r: (\mathbb{C} [t]_{< r})^n \to (\mathbb{C} [t]_{< r})^m$. Since $X_r$ is a constructible subset of $(\mathbb{C} [t]_{< r})^n \simeq \mathbb{C}^{nr}$, and $p_r$ has finite fibers on $X_r$, the estimate follows.
\end{proof}

\begin{cor}\label{trivcor}If $m < n$, $(\mathbb{A}^n_{\mathbb{C} \llp t \rrp} \setminus X)_1$
is nonempty.
\end{cor}

The following result shows the basic bound can be improved, as soon
as $X$ is not a linear subspace. By the degree of an irreducible affine variety $X$ over a field $k$ we mean the number of intersection points
when intersecting $X\otimes \bar k$ with a generic affine space over $\bar k$ of dimension equal to the codimension of $X$, for some algebraic closure $\bar k$ of $k$. The improved bound then reads as the trivial bound for $(m-1)$-dimensional varieties plus $r/d$, rounded up.

\begin{thm}\label{motivicBP} Let $X$ be an irreducible subvariety of
$\mathbb{A}^n_{\mathbb{C} \llp t \rrp}$ of dimension $m$ and degree $d$. 
Then,
for every positive integer $r$, one has
$$
n_r (X) \leq r (m - 1)  + \Bigl\lceil\frac{r}{d}\Bigr\rceil.
$$
\end{thm}

\begin{remark}\label{notonlyC}As already mentioned, Theorem \ref{motivicBP} remains valid with an identical proof  when  one replaces $\CC$ by any algebraically closed field of characteristic zero.
\end{remark}

\begin{remark}For the plane curve $X$ given by $y=x^d$, for every positive integer $r$, one has the equality
$$
n_r(X)= \Bigl\lceil\frac{r}{d}\Bigr\rceil.
$$
Thus, taking the product of $X$ with an affine space of dimension $m - 1$, one sees that the upper bound given in  Theorem \ref{motivicBP} is optimal for any value of $m$, $d$ and $r$.
\end{remark}

\subsection{} By simple projection and section arguments \`a la Lang-Weil, one reduces, completely analogously as in
\cite{pila_ast}, to the case of plane curves ($n = 2$ and $m = 1$).
For the sake of completeness let us provide some more details.
\begin{proof}[Reduction to the case $n = 2$ and $m = 1$ of Theorem \ref{motivicBP}]
Assume first $m = 1$ and $n > 2$.
Linear projections $\pi : \AA^n \to \AA^2$
are written in coordinates   as $x = \sum_{i = 1}^n a_i x_i$,
$y = \sum_{i = 1}^n b_i x_i$. For the $a_i$'s and $b_i$'s in a dense open subset $O$ of
$\AA^{2n}$, $\pi$ is surjective and $X$ and $\pi (X)$ have the same degree.
By Corollary \ref{trivcor}, $O (\mathbb{C}) := O(K)_1$ is non empty.
Thus take $\pi$ corresponding to some point in $O (\mathbb{C})$.
The number of points in the fibers
of
$\pi : X \to \Gamma = \pi (X)$
is finite and
$\pi (X_r)$ is contained in $\Gamma_{r}$, thus
the statement for $X$ follows from the one for $\Gamma$.
Now assume $m > 1$. By a similar argument, after projecting, one may assume $n = m +1$.
In the linear space   of hyperplanes
$H$ with equations $\sum_{i = 1}^n \alpha_i x_i = b$, $H \cap X$ is irreducible of degree $d$ outside
a closed subset $E$ of positive codimension.
Thus, by Corollary \ref{trivcor},
for some $\alpha_i$, $1 \leq i \leq n$ and $b_0$, all  in $\mathbb{C}$,
 the corresponding $H$ is not in $E$.
 Consider the pencil $H_b$ of hyperplanes
 $\sum_{i = 1}^n \alpha_i x_i = b$,
 $b \in \CC \llp t \rrp$. Since $H_{b_0}$ is not in $E$, the pencil intersects $E$ in at most $e$ points $b_j$
 such that
 $H_{b_j}$ lies in $E$. If $H_{b_j} = X$ for some $b_j$ we are done, so we may assume that,
 for such a $b_j$, $X \cap H_{b_j}$ is of dimension $\leq m - 1$, thus
 $n_r (X \cap H_{b_j}) \leq  (m - 1) r$ by the trivial estimate.
 For the other $b$'s one may apply the induction hypothesis, which gives
 $n_r(X\cap H_b)\leq  r (m - 2)  + \lceil\frac{r}{d}\rceil$.
 Since the linear form $\sum_{i = 1}^n \alpha_i x_i $ induces a constructible mapping
 $X_r \to \mathbb{C} [t]_{< r}$,
the statement follows by additivity of dimensions.
\end{proof}

\subsection{Hilbert functions}\label{hil}
Let $K$ be a field.
For $s \in \mathbb{N}$, we denote by
$K [x_0, \dots, x_n]_s$ the vector space of homogeneous polynomials of degree $s$.
Thus $K [x_0, \dots, x_n]_s$ is of dimension $L_n (s)$ over $K$.
Let $I $ be a homogeneous ideal of  $K [x_0, \dots, x_n]$ and set $I_s = I \cap K [x_0, \dots, x_n]_s$.
We set $H_I (s) = \dim K [x_0, \dots, x_n]_s /I_s$.
It is the Hilbert function of $I$.

Let $<$ be a monomial ordering on $K [x_0, \dots, x_n]$ in the sense of \cite{cox} Def 1, Ch. 1.2. Denote
by $\LT (I)$ the ideal generated by the leading terms for the ordering $<$ of elements of $I$, where the leading term of a homogeneous polynomial $f = \sum_i a_i x^i$ is the term among the nonzero $a_ix^i$ which is maximal for the ordering.
By \cite{cox} Prop 9, Ch. 9.3, $I$ and $\LT (I)$ have the same Hilbert function.

For each $0 \leq i \leq n$, set
$$\sigma_{I, i} (s) = \sum_{\alpha \in \Lambda_{n + 1} (s); x^{\alpha}\notin \LT (I)} \alpha_i.$$
Thus, $s H_I (s) = \sum_i \sigma_{I, i} (s)$.

Let $X$ be an irreducible subvariety of dimension $m$ and degree $d$ of $\mathbb{P}^n_K$
defined by a homogeneous ideal $I$.
Then, for $s$ large enough,
$H_I (s)$ is equal to $P_X (s)$ with
$P_X$ the Hilbert polynomial of $X$.
It is a polynomial of degree $m$, leading coefficient $d / m!$ and coefficients bounded in terms of $n$ and the degrees of generators of $I$.
As explained in \cite{broberg} and \cite{Marmon}, it follows there exist non-negative real numbers $a_{I, i}$, $i = 0, \dots, n$, such that
$$
\frac{\sigma_{I, i} (s)}{s H_I (s)} = a_{I, i} + O_{n, d} (1/ s)
$$
as $s \to \infty$.
Note that
$$
a_{I,0} + \cdots + a_{I, n} = 1.
$$

We shall need the following lemma of Salberger for $n=2$ and $m=1$.

\begin{lem}[Lemma 1.12 from \cite{sal}]\label{grevlex}Let $X$ be a closed equidimensional subscheme of dimension $m$
of $\mathbb{P}^n_K$. Assume $X$ intersects properly the hyperplane
$x_0 = 0$, that is,  no irreducible component of $X$ is contained in $x_0 = 0$. Let $<$ be the
monomial ordering
defined as follows:
$\alpha < \beta$ if $\vert \alpha \vert < \vert \beta \vert$
or if $\vert \alpha \vert = \vert \beta \vert$ and for some $i$,
$\alpha_i > \beta_i$ and $\alpha_j = \beta_j$, for $j < i$.
(That is, after reindexing the coordinates, $<$ is the
reverse graded lexicographic order.)
Then
$$
a_{I,1} + \cdots + a_{I, n} \leq \frac{m}{m + 1}.
$$
\end{lem}



\subsection{Proof of Theorem \ref{motivicBP} when $n = 2$ and $m = 1$}
Let  $K = \mathbb{C} \llp t \rrp$ and $X$ be an irreducible curve in
$\mathbb{A}^2_K$
of degree $d$.
Consider the map
$$
\iota : \begin{cases} K^2\to K^3 \\ (x,y)\mapsto (1,x,y) \end{cases}
$$
and the corresponding embedding
$$
\underline\iota :  \begin{cases} \mathbb{A}^2_{K} \longhookrightarrow
\mathbb{P}^2_{K} \\ (x, y) \mapsto [1: x: y] \end{cases}
$$
and let $I$ denote the homogeneous ideal of the closure of
$\underline\iota (X)$ in $\mathbb{P}^2_{K}$.
Let us form the leading term ideal $\LT (I)$ of $I$ for the monomial ordering $<$ of Lemma \ref{grevlex} for $n=2$.
Let $r\geq 0$ be given.
Fix a positive integer $\delta$.
Set
$$
M (\delta) = \Bigl\{j \in \mathbb{N}^{3} ; \vert j \vert = \delta, x^j\notin \LT (I) \Bigr\}.
$$
Set
$\mu = \# M (\delta) = H_I (\delta)$,  $\sigma_i = \sigma_{I,i} (\delta)$ for $i=0,1,2$, and $e =  (\mu-1)\mu/2$.
Let us write
$X' = X(K) \cap \cO_K^2$.
By Theorem \ref{GYT},
there exists a surjective  $\LPas^K$-definable function
$$
g:Y \subset \CC^s \times \cO_K  \to X'
$$
for some integer $s \geq 0$
such that for each
$\xi \in \CC^s$,
$g_\xi$ satisfies $T_{\mu}$ on $Y_\xi$.
Fix an integer $\alpha \geq 0$, and
let $B_\alpha$ be a closed ball of valuative radius $\alpha$ in
$\cO_K$.
Fix $\xi\in \CC^s$ and, for any choice of points $y_i$ for $1 \leq i \leq \mu$ in $(g_\xi (B_\alpha \cap Y_\xi))_r$,
consider the determinant
$$
\Delta = \det \Bigl(\iota(y_i)^j\Bigr)_{j \in M (\delta), 1 \leq i \leq \mu}.
$$
By Lemma \ref{detest} for $m=1$ and $n=2$ and Lemma \ref{prod},
we get that
\begin{equation}\label{upm}
\ord_t (\Delta) \geq  \alpha \cdot e.
\end{equation}
On the other hand, recall that $x \in \CC [t]$ belongs to $ \CC [t]_{<r}$ if and only if
$\deg (x) < r$, where $\deg$ stands for the degree in $t$,  and hence,
$$
\deg (\Delta)  \leq   (r-1) (\sigma_1 + \sigma_2).
$$
Thus, if $\Delta \not= 0$,
\begin{equation}\label{downm}
\ord_t (\Delta)  \leq    (r-1) (\sigma_1 + \sigma_2).
\end{equation}
By putting together (\ref{upm}) and (\ref{downm}),
it follows that if
\begin{equation}\label{boundonrhom}
\alpha >  (r-1) (\sigma_1  +\sigma_2)/e,
\end{equation}
then $\Delta = 0$.
For such an $\alpha$,
note that the matrix
$$
 A = \Bigl(y_i^j\Bigr)
$$
with $j$ running over $ M (\delta)$ and $y_i$ in $g_\xi (B_\alpha \cap Y_\xi)_r$ for $i=1,\ldots,\mu$,
has rank $\leq \mu - 1$.
Hence, by the arguments in the proof of Lemma 1 in \cite{BP} (which are quite similar to those from the proof of Lemma \ref{basicest}), there exists a nonzero polynomial $H$ in two variables with coefficients in $\CC [t]$ and exponents
in $M (\delta)$ which vanishes at all the $y_i$, and thus at all points of
$g_\xi (B_\alpha\cap Y_\xi)_r$.
Note that $H$ does not vanish identically on $X$ since its exponents lie
in $M (\delta)$ and that its degree is at most $\delta$.

Recall that $r>0$ is given and we want to prove that
$n_r (X) \leq \lceil\frac{r}{d}\rceil$. We will prove this bound by choosing $\delta=\delta(r)$ following \cite{Marmon}.
By properties of Hilbert polynomials recalled in Section \ref{hil}, we have
$$\sigma_i = a_i d\delta^{2} + O_{d} (\delta),$$
$$
\mu = d  \delta + O_{d} (1),
$$
and thus
$$
e = \frac{d^{2}}{2} \delta^{2} + O_{d} (\delta),
$$
where the $O_d$ notation is for $\delta$ going to $+\infty$.
Thus,
$$
\frac{\sigma_i}{e} =  \frac{2 a_i}{d} + O_{ d} (\delta^{-1}).
$$
By Lemma \ref{grevlex} we find
$$
\frac{ \sigma_1+\sigma_2}{e} \leq  \frac{1}{d} + O_{ d} (\delta^{-1}).
$$
Hence, there exist integers $\delta>0$ and  $\alpha > 0$, both depending on $r$, such that
$$
 (r-1)\frac{ \sigma_1+\sigma_2 }{e} < \alpha \leq \Bigl\lceil\frac{r}{d}\Bigr\rceil.
$$
Now we are ready to bound $n_r (X')$, using this choice of $\delta $ and $\alpha$.
Note that $X_r$ is Zariski closed in $(\CC [t]_{< r})^2\simeq \CC^{2r}$
being an intersection of Zariski closed subsets.

Write
$$
p : \CC^s \times \cO_K \to \CC^s \times \cO_K/\cM_K^\alpha \simeq \CC^{ s +  \alpha}
$$
for the projection, where $\cM_K$ is the maximal ideal of $\cO_K$. By Remark \ref{rem:inj}, we may suppose that the projection $\CC^s \times \cO_K  \to \cO_K$
is finite to one on $Y$, and that $Y$ is a cell over $\CC^s$. This implies
\begin{equation}\label{dimpY}
\dim (p(Y) ) \leq \alpha,
\end{equation}
where $p(Y)$ is considered as a constructible subset of $\CC^{ s +  \alpha}$.
The set $\widehat{X}$ of all $(w, x)$ in
$$
p(Y)\times X_r
$$
such that there exists $y\in Y$ with $p(y)=w$ and $g(y)=x$ is a constructible subset of $\CC^{s+\alpha+2r}$ by Proposition \ref{cons}. Since $g$ is surjective, the projection $\widehat{X}\to X_r$ is also surjective. By construction, the projection $\widehat{X}\to  p(Y)$ has finite fibers  (of size at most $d\delta$ by B\'ezout's Theorem):
indeed, a point in $p (Y)$ corresponds to the choice of $\xi \in \mathbb{C}^s$ together with a ball $B_{\alpha}$ as above, and we have shown that the fiber
of
the projection $\widehat{X}\to  p(Y)$ over this point is contained in the intersection of $X$ with the zero locus of   a polynomial $H$ in two variables,
with
coefficients in $\mathbb{C} [t]$ and
 degree
at most $\delta$,   that does not vanish identically on the curve $X$.
It follows that $ n_r (X')\leq \dim (\widehat{X}) \leq \dim (p(Y)) \leq \alpha\leq \lceil r/d \rceil$.
\qed

\medskip

The following proposition is deduced from a result of \cite{CLR} on quantifier elimination in an expansion of $\cL$ which includes more auxiliary sorts, namely all the $\cO_K/\cM_K^\alpha$ for integers $\alpha>0$ (not to be confused with the  $K_n = \cO_K/(n\cM_K) \simeq\CC$ for $n>0$).

\begin{prop}\label{cons}  
Let $X\subset \CC^s\times \cO_K^{n+m}$ be an $\LPas^K$-definable set, let $\alpha>0$ and $r>0$ be integers and let
$$
p: \CC^s\times \cO_K^n\times \cO_K^m\to \CC^s\times (\cO_K/(t^\alpha))^n\times \cO_K^m  
$$
be the projection. 
Write $p(X)_r$ for the intersection of $p(X)$ with $\CC^s\times (\cO_K/(t^\alpha))^n\times (\cO_K^m)_r$.
Then $p(X)_r$, seen as subset of $\CC^{s + \alpha n + mr}$, is definable in the ring language with coefficients from $\CC$. 
\end{prop}
\begin{proof}
Let $\LPas'$ be the language $\LPas$ enriched with the auxilliary sorts $\cO_K\bmod (t^{\alpha})$ for each integer $\alpha>0$, (higher order) angular component maps
$$
\mathrm{ac}_\alpha: K \to \cO_K\bmod (t^{\alpha})
$$
sending nonzero $x$ to $x t^{-\ord x} \bmod (t^{\alpha})$ and zero to zero, and the bijections from $\cO_K\bmod (t^{\alpha})$ to $\CC^\alpha$ sending $\sum_{i=0}^{\alpha-1} x_i t^i $ to the tuple $(x_i)_i$. (The maps $\mathrm{ac}_\alpha$ should not be confused with the maps $\ac_n$ introduced before.)
One has quantifier elimination for all sorts in the language $\LPas'$ by \cite{CLR}[Thm. 4.2].
It follows that $p(X)$ is $\LPas'$-definable without quantifiers.
Moreover, in the variables running over $\cO_K^{m}$, finitely many polynomials $f_i(x)$ over $K$ can occur in the formula describing $p(X)$, and we may suppose they occur as arguments of $\ac_\alpha$ for some $\alpha$ and of $\ord$. Note that, by the geometry of definable sets as summarized by the cell decomposition result, the $\ord(f_i)$ take only finitely many values on $(\cO_K^m)_r$.
Now the lemma follows by syntactical analysis of quantifier free formulas describing $p(X)$, in relation with the extra condition that $x$ lies in $(\cO_K^m)_r$.
\end{proof}

\begin{remark}\label{remdim}
In fact, the inequality (\ref{dimpY}) 
holds in a wider generality.
Consider an $\LPas^K$-definable set $X\subset \CC^s\times \cO_K^{n}$ and suppose that the projection
$$
\pi:\CC^s\times \cO_K^{n}\to \cO_K^n
$$
is finite to one on $X$, where $K=\CC\llp t \rrp$. Suppose further that $\pi(X)$ is of dimension $m$.
Let $\alpha>0$ be an integer and 
write
$$
p : \CC^s  \times \cO_K^n \to  \CC^{s+  \alpha n},
$$
for the projection which is $\cO_K\to \cO_K/(t^\alpha) \simeq \CC^{\alpha}$ on the last $n$ coordinates.
Then one has
\begin{equation}\label{dimpX}
\dim (p(X) ) \leq \alpha m.
\end{equation}
Indeed, (\ref{dimpX}) is easy to show when $X$ is a cell, and follows by the cell decomposition theorem \ref{cda} in general.
\end{remark}

%
\subsection{An observation on the size of motivic transcendental parts}
When $\cL=\Lan^K$ with $K=\CC\llp t \rrp$, one may wonder whether one can bound $(X^{\rm trans})_r$ in terms of $r>0$, when $X^{\rm trans}$ is the transcendental part of a definable subset $X$ in $\cO_K^n$, and with notation from Section \ref{nr}.
As usual, $X^{\rm trans}$  is $X\setminus X^{\rm alg}$, where the algebraic part $X^{\rm alg}$ of $X$ is defined as the set of those points $x\in X$ through which there exists a semi-algebraic (namely $\LPas^K$-definable) $S$ of dimension $1$ such that $X\cap S$ is locally around $x$ of dimension $1$.


A first idea would be to try bounding the dimension $n_r(X^{\rm trans})$ in terms of $r$, but such bounds are useless in view of Proposition \ref{obs}.

\begin{prop}\label{obs}
Let $X\subset \cO_K^n$ be $\cL$-definable. Then, for any $r>0$ and any algebraic curve $C\subset \CC^{nr}$, the intersection of $C$ with $(X^{\rm trans})_r\subset \CC^{nr}$ is finite.
\end{prop}
\begin{proof}
Suppose for contradiction that there is an algebraic curve $C\subset \CC^{nr}$ with infinite intersection with $X^{\rm trans}_r$.
Let us write $C(K)$ for the subset of $K^{nr}$ of $K$-rational points on $C$.
Further, let us write $S$ for the image of $C(K)\cap \cO_K^{nr}$ under the projection
$$
\pi :\begin{cases}\cO_K^{nr}\to \cO_K^n \\ (x_{1,0},\ldots,x_{1,r-1},\ldots, x_{n,0},\ldots,y_{n,r-1} ) \mapsto
(\sum_{i=0}^{r-1}t^i x_{j,i})_{j=1}^n.\end{cases}
$$
Since $C(K)\cap \cO_K^{nr}$ is an $\LPas^K$-definable set of dimension at most $1$, and since $S$ is its image under an $\LPas^K$-definable function, the dimension of $S$ is at most equal to $1$. Since $S\cap X$ contains an infinite subset of $C$, the dimension of $S\cap X$ equals $1$.
Here, we have used the natural inclusion
$$
C \subset \CC^{nr}\simeq (\mathbb{C} [t]_{< r})^n \subset \CC\llb t \rrb^n = \cO_K^n.
$$
Moreover, $S\cap X$ is of local dimension $1$ at all but finitely many of its points, since it is an $\cL$-definable set. Hence, $X^{\rm trans}$ is contained in the union of a finite set with $X\setminus S$.
Since $(X\setminus S)_r=X_r\setminus S_r $ and since $S_r$ contains $C$,
$X^{\rm trans}_r$ cannot have infinite intersection with $C$.
\end{proof}

Finally, let us mention that it seems quite difficult  to give sharp bounds on the size of the set $X^{\rm trans}_r$ in terms of $r$ in general, for $X$ of large dimension. Under some extra conditions on $X$, like with some non-archimedean analogues of restricted Pfaffians instead of the full subanalytic language on $K$, one may hope there exist results for (low-dimensional) definable sets, similar to e.g. the results in \cite{PilaP} for real Pfaffian curves.

\subsection{From $\CC (t)$ to $\mathbb{F}_q (t)$}

In this section we shall discuss related results over $\mathbb{F}_q (t)$. Although our methods are restricted to characteristic zero, due to our limited understanding of the structure of definable sets over henselian fields of positive characteristic, it is possible to
use standard methods to deduce from Theorem \ref{motivicBP}
asymptotic bounds for the number of rational points in $\FF_q [t]$ of  bounded degree. This provides in particular
a partial answer to a question raised by Cilleruelo and Shparlinski in \cite{CS} Problem 9, about possible analogues
over $\mathbb{F}_q (t)$  of the Bombieri-Pila bound.
As noticed in Remark \ref{cohenff}, note however that an analogue of S. D. Cohen's bound holds over $\mathbb{F}_q (t)$.

Let $R$ be an algebra essentially of finite type over $\mathbb{Z}$, i.e. the localization of a finitely generated $\mathbb{Z}$-algebra.
We assume $R$ is an integral domain of characteristic zero and we denote by $K$ its fraction field.

We consider the category $\Fie_R$ of ring morphisms $R \to F$ with $F$ a field, i.e. the category of field endowed with an $R$-algebra structure.
If $\cX$ is an $R$-scheme, and  $R \to F$ a ring morphism, we denote by $\cX \otimes F$ the $F$-scheme obtained by base change to $F$.
We consider the affine space
$\AA^n_{R [t]}$ and $X$ a closed subscheme whose ideal is generated by polynomials
$f_1, \dots, f_s \in R[t][X_1, \dots, X_n]$.

For any positive integer $r$, we denote by  $F[t]_{< r}$ the set of polynomials with coefficients in $F$ and degree $<r$.
We identify $F[t]_{< r}$ with $F^r$ and $(F[t]_{< r})^n$ with $F^{rn}$.

The following lemma is classical, we provide a proof for the sake of completeness.
\begin{lem}\label{represent}
Let $r$ be a positive integer. The functor
$
X_r : \Fie_R \to \mathrm{Sets}
$
sending $R \to F$ to
$$
X_{F, r} := X (F[t]) \cap (F[t]_{< r})^n \subset F^{rn}
$$
is representable by a closed subscheme
$\cX_r$ of $\AA^{rn}_R$. In particular for any $F$ in $\Fie_R$
we have a natural identification of
$X_{F, r}$ with $\cX_r (F)$ inside $R^{rn}$.
\end{lem}

\begin{proof}After inserting the polynomials $x_i (t) = \sum_{0 \leq \gamma <r} a_{i \gamma}t^\gamma$
into the system of equations
$f_j (x_1, \dots, x_n)$ and developing, one gets a system of polynomial equations in the variables $a_{i \gamma}$.
The corresponding closed subvariety
$\cX_r$ of $\AA^{rn}_R$
represents the functor $X_r$.
\end{proof}

If $k$ is a finite field, we denote by $p_k$ its characteristic and $q_k$ its cardinality.
We shall use the following lemma, which is a consequence from statements in \cite{CVM}  based on the Lang-Weil estimate.

\begin{lem}\label{langweil}

Let $\cZ$ be a closed subscheme of $\AA^m_R$.
Let $n = \dim (\cZ \otimes K)$.
There exists positive integers $p_0$, $C$ and $M$, such that, for
any $R \to k$ in $\Fie_R$ with $k$ finite,
if $p_k > p_0$ and $\cZ (k) \not= \emptyset$,
then for some $\delta \leq n$ and some $\mu \in \{1, \dots, M\}$,
\[
\vert \# \cZ (k)- \mu q_k^\delta \vert \leq C q_k^{\delta - \frac{1}{2}}.
\]
\end{lem}

\begin{proof}It follows from  Proposition 3.3 and Proposition 4.9  in \cite{CVM} that there exist positive integers
$C$ and $M$ such that, for any $R \to k$ in $\Fie_R$ with $k$ finite,
if
$\cZ (k) \not= \emptyset$,
$\vert \# \cZ (k)- \mu q_k^\delta \vert \leq C q_k^{\delta - \frac{1}{2}}$
for some $\mu \in \{1, \dots, M\}$ and $\delta$ the dimension of the Zariski closure of  $\cZ (k)$ in $\AA^m_k$.
In particular,
$\delta \leq \dim (\cZ \otimes k)$.
Since, for some $p_0$,
$\dim (\cZ \otimes k) = \dim (\cZ \otimes K)$ whenever $p_k \geq p_0$, the  statement  follows.
\end{proof}

Now we can state our result, which provides a partial answer to Problem 9 in \cite{CS}.

\begin{thm}\label{ffBP}Let $R$ be an algebra essentially of finite type over $\mathbb{Z}$ and assume $R$ is an integral domain of characteristic zero.
Let   $K$ be the fraction field of $R$ and $\overline K$ an algebraic closure of $K$.
Let $X$ be a closed subscheme of $\AA^n_{R [t]}$. Assume $X \otimes {\overline K} \llp t \rrp$ is irreducible
of dimension $m$ and degree $d$.
Fix a positive integer $r$.
There exist positive integers $p_0$, $C$ and $M$, such that, for
any $R \to k$ in $\Fie_R$ with $k$ finite,
if $p_k > p_0$ and $X_{k, r} \not= \emptyset$,
then
\[
\vert \# X_{k, r}- \mu q_k^\delta \vert \leq C q_k^{\delta - \frac{1}{2}}
\]
for some $\delta \leq r (m - 1)  + \lceil\frac{r}{d}\rceil$ and some $\mu \in \{1, \dots, M\}$.
\end{thm}

\begin{proof}From Theorem \ref{motivicBP}, which stills hold when replacing $\CC$ by $\overline K$ by Remark \ref{notonlyC},
it follows that $\dim (\cX_r (\overline K)) \leq r (m - 1)  + \lceil\frac{r}{d}\rceil$.
Since $\dim (\cX_r \otimes  K) = \dim (\cX_r (\overline K))$,
the statement is then a direct consequence of Lemma \ref{langweil} applied to
$\cZ = \cX_r$.
\end{proof}

\begin{remark}\label{cohenff}When $d > 1$,
Theorem \ref{ffBP} provides a non-trivial improvement on the ``trivial'' bound with $\delta \leq rm$.
Note however that, using the function field version of the large sieve inequality due to Hsu \cite{Hsu} instead of the one used in \cite{serre}, one can easily adapt the  arguments given in \cite{serre}
to get the following
function field analogue of S. D. Cohen's result in \cite{cohen}:  if $X$ is an irreducible subvariety of
$\mathbb{A}^n_{\mathbb{F}_q \llp t \rrp}$ of dimension $m$ and degree $d\geq 2$,
then $\# X_{r} = O (r q^{r(m - \frac{1}{2})})$, with $X_{r} := X (\mathbb{F}_q[t]) \cap (\mathbb{F}_q[t]_{< r})^n$. 
\end{remark}

The following question seems natural:

\begin{question}Does Theorem \ref{motivicBP} still hold when $\mathbb{C}$ is replaced by an algebraically closed field of positive characteristic?
\end{question}

\bibliographystyle{amsplain}

\end{spacing}
\end{document}